\documentclass[reqno]{amsart}
\usepackage{enumerate}
\usepackage{tabto}
\usepackage[mathscr]{euscript}
\usepackage{xcolor}
\usepackage{layout}
\usepackage{fancyhdr}
\usepackage{array}
\usepackage{amsfonts}
\usepackage{amsmath}
\usepackage{amssymb}
\usepackage{mathtools}
\usepackage{graphicx}
\usepackage{bm}
\usepackage{enumitem}
\usepackage{caption} 
\usepackage{color}
\usepackage{kotex}
\usepackage{csquotes}
\usepackage{bookmark}
\usepackage{float}
\usepackage{multirow}
\usepackage[square,numbers,sort&compress]{natbib}
\usepackage{hyperref}
\hypersetup{colorlinks=true,linkcolor=blue,citecolor=red}
\allowdisplaybreaks

\def\Xint#1{\mathchoice
{\XXint\displaystyle\textstyle{#1}}%
{\XXint\textstyle\scriptstyle{#1}}%
{\XXint\scriptstyle\scriptscriptstyle{#1}}%
{\XXint\scriptscriptstyle\scriptscriptstyle{#1}}%
\!\int}
\def\XXint#1#2#3{{\setbox0=\hbox{$#1{#2#3}{\int}$ }
\vcenter{\hbox{$#2#3$ }}\kern-.6\wd0}}

\def\dashint{\Xint-}

\newtheorem{theorem}{Theorem}[section]
\newtheorem{lemma}[theorem]{Lemma}

\theoremstyle{definition}
\newtheorem{definition}[theorem]{Definition}

\numberwithin{equation}{section}

\newcommand{ \mr }{ \mathbb{R} }

\newcommand{\iints}[1]{{\int\hspace{-0.28cm}\int_{#1}}}
\newcommand{\iintss}{{\int\hspace{-0.28cm}\int}}
\newcommand{ \miints }{{\iintss\hspace{-0.56cm} -\hspace{-0.15cm}-}}
\newcommand{\miint}[1]{{\miints_{\hspace{-0.13cm}#1}}}

\begin{document}
\title[Bounded solutions and interpolative gap bounds]{Bounded solutions and interpolative gap bounds for degenerate parabolic double phase problems}

\author{Bogi Kim}\address{Department of Mathematics, Kyungpook National University, Daegu, 41566, Republic of Korea} \email{rlaqhrl4@knu.ac.kr} \author{Jehan Oh}\address{Department of Mathematics, Kyungpook National University, Daegu, 41566, Republic of Korea} \email{jehan.oh@knu.ac.kr}

\date{\today.}
\thanks{This work is supported by National Research Foundation of Korea (NRF) grant funded by the Korea government [Grant Nos. RS-2023-00217116, RS-2025-00555316, RS-2025-25415411, and RS-2025-25426375].}

\begin{abstract}
We establish gradient higher integrability results for weak solutions to degenerate parabolic equations of double phase type
$$
u_t-\operatorname{div} \left(|Du|^{p-2}Du + a(x,t)|Du|^{q-2}Du\right)=0 
$$
in $\Omega_T \coloneq \Omega\times (0,T)$, where $a(\cdot)\in C^{\alpha,\frac{\alpha}{2}}(\Omega_T)$.
For bounded solutions, we prove that the result holds under the gap condition
$$
q \leq p + \alpha.
$$
Moreover, for solutions with
$$
u\in C(0,T;L^s(\Omega)), \quad s \geq 2,
$$
we obtain higher integrability under the gap condition
$$
q \leq p + \frac{s\alpha}{n+s}.
$$
These results provide an interpolation between the gap bounds in the parabolic double phase setting.
\end{abstract}
\maketitle

\section{\bf Introduction}\label{section 1}
We study parabolic double phase equations of the form
\begin{equation}    \label{eq : main equation}
    u_t - \operatorname{div} \mathcal{A}(z,Du)=0 \qquad \text{in} \ \, \Omega_T\coloneq \Omega\times (0,T),
\end{equation}
where $\Omega \subset \mr^n$ ($n\geq 2$) is a bounded open set and $\mathcal{A}:\Omega_T \times \mr^{n}\rightarrow \mr^{n}$ is a Carath\'{e}odory vector field satisfying the following coercivity and growth bounds: there exist constants $0<\nu\leq L <\infty$ such that
\begin{equation}    \label{cond : double phase bounded condition of integrand}
    \mathcal{A}(z,\xi)\cdot \xi\geq \nu (|\xi|^p+a(z)|\xi|^q)\quad \text{and}\quad |\mathcal{A}(z,\xi)|\leq L(|\xi|^{p-1}+a(z)|\xi|^{q-1})
\end{equation}
for all $z \in \Omega_T$ and $\xi \in \mr^n$, where $2\leq p <q <\infty$.
We note that the corresponding model equation is
\begin{equation}    \label{eq : model equation}
u_t-\operatorname{div} \left(|Du|^{p-2}Du + a(z)|Du|^{q-2}Du\right)=0 \qquad \text{in} \ \, \Omega_T.
\end{equation}

The parabolic equation \eqref{eq : model equation} is the evolutionary counterpart of the elliptic double phase functional
$$
W^{1,1}(\Omega) \ni w \mapsto \mathcal{P}(w,\Omega)\coloneq \int_\Omega \left[\frac{1}{p}|Dw|^p+\frac{1}{q}a(x)|Dw|^q\right] dx,
$$
where $1<p\leq q$ and $0\leq a(\cdot)\in C^{\alpha}(\Omega)$ for some $\alpha\in(0,1]$.
Its Euler-Lagrange equation is given by
$$
-\operatorname{div}(|Du|^{p-2}Du + a(x)|Du|^{q-2}Du)=0 \qquad \text{in} \ \, \Omega,
$$
which is associated with an elliptic double phase model that was originally introduced in \cite{Zhikov1986,Zhikov1993,Zhikov1995,Zhikov1997} as an example exhibiting the Lavrentiev phenomenon and illustrating homogenization in strongly anisotropic materials. Moreover, variants of the double phase problem arise naturally in applied sciences, including transonic flows \cite{Bahrouni2019}, quantum physics \cite{Benci2000}, stationary reaction-diffusion systems \cite{Cherfils2005stationary}, image denoising and processing \cite{Kbiri2014,Charkaoui2024,Chen2006,Harjulehto2013,Harjulehto2021,Fang2010}, and heat diffusion in materials with heterogeneous thermal properties \cite{Arora2023}, and so on. To obtain regularity results for weak solutions of elliptic double phase problems, \cite{Colombo2015a} and \cite{Mingione2021} imposed conditions linking the closeness of $p$ and $q$ with the H\"{o}lder exponent $\alpha$ of the modulating coefficient $a(\cdot)$. In particular, under the gap bound condition  
\begin{equation}\label{cond : gamma=2 gap bound condition in elliptic double phase problem}
\frac{q}{p}\leq 1+\frac{\alpha}{n},
\end{equation}
it has been shown in \cite{Baroni2018,Colombo2015,Esposito2004} that a weak solution $u$ and its gradient $Du$ are H\"{o}lder continuous. Also, when $u$ satisfies the conditions
\begin{equation}\label{cond : gamma=infty gap bound condition in elliptic double phase problem}
u\in L^\infty(\Omega)\quad \text{and}\quad q\leq p +\alpha,
\end{equation}
the same conclusion has been established in \cite{Baroni2018,Colombo2015a}. Moreover, Baroni-Colombo-Mingione \cite{Baroni2018} have established that if 
$$
u\in C^{0,\gamma}(\Omega) \quad \text{and}\quad q\leq p +\frac{\alpha}{1-\gamma}\ \, \text{with}\ \, \gamma\in(0,1),
$$
then the gradient of $u$ is H\"{o}lder continuous. This shows that by imposing stronger regularity assumptions on $u$, one may relax the gap bound condition while still ensuring regularity results for $u$. Indeed, in \cite{Ok2020}, under the assumption
\begin{equation}\label{cond : gamma>1 gap bound condition in elliptic double phase problem}
u\in L^{\gamma}_{\operatorname{loc}}(\Omega)\quad \text{and}\quad q\leq p +\frac{\gamma \alpha}{n+\gamma},
\end{equation}
where $p$ and $\gamma$ satisfy
$$
1<p<n\quad \text{and} \quad \gamma>\frac{np}{n-p},
$$
it was shown that a local quasi-minimizer $u$ of $\mathcal{P}$ is locally H\"{o}lder continuous, which in turn led to an interpolation of the gap bound conditions. Furthermore, under either \eqref{cond : gamma=2 gap bound condition in elliptic double phase problem} or \eqref{cond : gamma=infty gap bound condition in elliptic double phase problem}, various regularity results have also been established. For instance, Baroni-Colombo-Mingione \cite{Baroni2015}, Ok \cite{Ok2017,Ok2020} have obtained Harnack's inequality and H\"{o}lder continuity for weak solutions, while Baasandorj-Byun-Oh \cite{Baasandorj2020}, De Filippis-Mingione \cite{DeFilippis2019} and Colombo-Mingione \cite{Colombo2016} have established Calder\'{o}n-Zygmund type estimates. Furthermore, several regularity results concerning the elliptic double phase problems have been obtained in \cite{Byun2021,Byun2017,Byun2020,Byun2021a,Haestoe2022,Haestoe2022a}.

Next, to study the regularity of weak solutions to parabolic double phase problems, it is assumed in Kim-Kinnunen-Moring \cite{2023_Gradient_Higher_Integrability_for_Degenerate_Parabolic_Double-Phase_Systems} and Kim-Kinnunen-S\"{a}rki\"{o} \cite{Wontae2023a} that the conditions
\begin{equation}\label{cond : s=2 gap bound condition in parabolic double phase problem}
    u\in C(0,T;L^2(\Omega)),\quad a(\cdot) \in C^{\alpha,\frac{\alpha}{2}}(\Omega_T) \quad \text{and}\quad q\leq p + \frac{2\alpha}{n+2}
\end{equation}
hold. Here, $a(\cdot) \in C^{\alpha,\frac{\alpha}{2}}(\Omega_T)$ means that $a(\cdot)\in L^\infty (\Omega_T)$ and there exists a H\"{o}lder constant $[a]_\alpha\coloneq [a]_{\alpha,\frac{\alpha}{2};\Omega_T}>0$ such that
$$
|a(x_1,t_1)-a(x_2,t_2)|\leq [a]_\alpha \max\left\{|x_1-x_2|^\alpha,|t_1-t_2|^\frac{\alpha}{2}\right\}
$$
for all $x_1,\,x_2\in\Omega$ and $t_1,\,t_2\in(0,T)$. Under the assumption \eqref{cond : s=2 gap bound condition in parabolic double phase problem}, they have established the higher integrability results and energy estimates of weak solutions. Also, the existence theory of weak solutions to the above problem is addressed in \cite{Chlebicks2019}, with further discussions in \cite{Wontae2023a,Singer2016}. Kim \cite{Wontae2023b} and Kim-S\"{a}rki\"{o} \cite{Wontae2024} have investigated higher integrability results and Calder\'{o}n-Zygmund type estimates for singular parabolic double phase systems. Furthermore, Buryachenko-Skrypnik \cite{Buryachenko2022} have established local continuity and Harnack's inequality for the parabolic double phase equations, and Kim-Oh \cite{Kim2024}, Kim-Oh-Sen \cite{Kim2025} and Sen \cite{Sen2025} have studied regularity results for the parabolic multi-phase problems.

We now introduce the definition of a weak solution for the parabolic double phase problem.
\begin{definition}
    A function $u:\Omega_T\rightarrow \mr$ with
    $$
    u\in C(0,T;L^2(\Omega))\cap L^q(0,T;W^{1,q}(\Omega))
    $$
    is a weak solution to \eqref{eq : main equation} if
    $$
    \iints{\Omega_T} -u\cdot \varphi_t +\mathcal{A}(z,Du)\cdot D\varphi\, dz=0 
    $$
    for every $\varphi\in C_0^\infty(\Omega_T)$.
\end{definition}

We remark that, by employing the parabolic Lipschitz truncation method as in \cite{Wontae2023a}, one may assume that $u$ belongs to a more natural function space
$$
u\in C(0,T;L^2(\Omega)) \cap L^1(0,T;W^{1,1}(\Omega))
$$
with
$$
\iints{\Omega_T} H(z,|Du|)\, dz=\iints{\Omega_T} (|Du|^p +a(z)|Du|^q)\, dz <\infty,
$$
instead of assuming
$$
    u\in C(0,T;L^2(\Omega))\cap L^q(0,T;W^{1,q}(\Omega)).
$$

In this paper, we first consider bounded weak solutions to \eqref{eq : main equation}. We assume that the non-negative coefficient function $a:\Omega_T\rightarrow \mr^+$ and a weak solution $u$ satisfy
\begin{equation}    \label{cond : main assumption with infty}
    u\in L^\infty(\Omega_T),\quad a(\cdot) \in C^{\alpha,\frac{\alpha}{2}}(\Omega_T)\quad \text{and}\quad q\leq p+\alpha
\end{equation}
for some $\alpha\in(0,1]$. This condition is dimensionless and coincides with the gap bound condition \eqref{cond : gamma=infty gap bound condition in elliptic double phase problem} for bounded solutions to the elliptic double phase problem. Kim-Moring-S\"{a}rki\"{o} \cite{Wontae2025} have established that, under the condition \eqref{cond : main assumption with infty}, weak solutions are locally H\"{o}lder continuous. Assuming  \eqref{cond : main assumption with infty}, we obtain a higher integrability result.
In what follows, $Q_r(z_0)\coloneq B_r(x_0)\times (t_0-r^2,t_0+r^2)$ denotes a parabolic cylinder, where $z_0=(x_0,t_0)$.
For simplicity, we write the collection of parameters as
$$
\operatorname{data}_b \coloneq (n,p,q,\alpha,\nu,L,\operatorname{diam}(\Omega),[a]_\alpha,\|u\|_{L^\infty(\Omega_T)}),
$$
and we denote $H(z,\varkappa) \coloneq \varkappa^p +a(z)\varkappa^q$ for $\varkappa\geq 0$ and $z\in\Omega_T$.
\begin{theorem}\label{thm : main theorem with infty}
    Assume that \eqref{cond : main assumption with infty} is satisfied and let $u$ be a weak solution to \eqref{eq : main equation} with \eqref{cond : double phase bounded condition of integrand}. Then there exist constants $\varepsilon_0=\varepsilon_0(\operatorname{data}_b)>0$ and $c=c(\operatorname{data}_b,$ $\|a\|_{L^\infty(\Omega_T)})>1$ such that
    $$
    \begin{aligned}
        &\miint{Q_r(z_0)} H(z,|Du|)^{1+\varepsilon}\, dz\leq c\left(\miint{Q_{2r}(z_0)} H(z,|Du|)\,dz\right)^{1+\frac{\varepsilon q}{2}}+c
    \end{aligned}
    $$
    for every $Q_{2r}(z_0)\subset \Omega_T$ and $\varepsilon\in (0,\varepsilon_0)$.
\end{theorem}

Next, we further consider the case in which an interpolation between \eqref{cond : s=2 gap bound condition in parabolic double phase problem} and \eqref{cond : main assumption with infty} is given by the following conditions:
\begin{equation}    \label{cond : main assumption with s}
u\in C(0,T;L^s(\Omega)),\quad a(\cdot)\in C^{\alpha,\frac{\alpha}{2}}(\Omega_T)\quad \text{and}\quad q\leq p+\frac{s\alpha}{n+s}
\end{equation}
for some $s\in [2,\infty)$ and for some $\alpha\in(0,1]$.  Our final goal is to prove the gradient higher integrability result for a weak solution $u$ to \eqref{eq : main equation} under the assumption \eqref{cond : main assumption with s}. For this, we write
$$
\operatorname{data}_s \coloneq (n,p,q,s,\alpha,\nu,L,\operatorname{diam}(\Omega),[a]_\alpha,\|u\|_{C(0,T;L^s(\Omega))}).
$$
\begin{theorem}\label{thm : main theorem with s}
    Assume that \eqref{cond : main assumption with s} is satisfied and let $u$ be a weak solution to \eqref{eq : main equation} with \eqref{cond : double phase bounded condition of integrand}. Then there exist constants $\varepsilon_0=\varepsilon_0(\operatorname{data}_s)>0$ and $c=c(\operatorname{data}_s,$ $\|a\|_{L^\infty(\Omega_T)})>1$ such that
    $$
    \begin{aligned}
        &\miint{Q_r(z_0)} H(z,|Du|)^{1+\varepsilon}\, dz\leq c\left(\miint{Q_{2r}(z_0)} H(z,|Du|)\,dz\right)^{1+\frac{\varepsilon q}{2}}+c
    \end{aligned}
    $$
    for every $Q_{2r}(z_0)\subset \Omega_T$ and $\varepsilon\in (0,\varepsilon_0)$.
\end{theorem}

We remark that $\frac{s\alpha}{n+s}= \frac{2\alpha}{n+2}$ when $s= 2$ and that $\frac{s\alpha}{n+s}\nearrow \alpha$ as $s\rightarrow \infty$. This means that Theorem \ref{thm : main theorem with s} yields the interpolation of gap bound conditions in Theorem \ref{thm : main theorem with infty} and \cite[Theorem 1.1]{2023_Gradient_Higher_Integrability_for_Degenerate_Parabolic_Double-Phase_Systems}.
Moreover, this result provides a justification for the naturalness of the bound established in the previous study \cite{2023_Gradient_Higher_Integrability_for_Degenerate_Parabolic_Double-Phase_Systems}.
On the other hand, we impose a strong assumption on $u$, related to the function space $C(0,T;L^2(\Omega))$ and hence to the time derivative of $u$. Unlike \cite{Ok2020}, this condition concerns the time variable rather than the spatial one, which distinguishes the parabolic case from the elliptic one.

However, the modified condition causes difficulties throughout the proof. In contrast to \cite{2023_Gradient_Higher_Integrability_for_Degenerate_Parabolic_Double-Phase_Systems}, we distinguish the $p$-intrinsic and $(p,q)$-intrinsic cases by imposing
$$
K\lambda^p\geq \sup a(\cdot) \lambda^q \quad \text{and}\quad K\lambda^p \leq \sup a(\cdot) \lambda^q,
$$
respectively, for some $K>1$. These conditions simplify the proof of the parabolic Sobolev-Poincar\'{e} type inequalities in Section \ref{section 4}. We shall consider the following three cases:
\begin{enumerate}
    \item $\displaystyle K\lambda^p\geq \sup_{Q_{10\rho}(z)} a(\cdot)\lambda^q$,
    \item $\displaystyle K\lambda^p\leq \sup_{Q_{10\rho}(z)} a(\cdot)\lambda^q\quad$and$\quad\displaystyle \sup_{Q_{10\rho}(z)} a(\cdot)\gtrsim  \rho^\alpha$,
    \item $\displaystyle K\lambda^p\leq \sup_{Q_{10\rho}(z)} a(\cdot)\lambda^q\quad$and$\quad\displaystyle \sup_{Q_{10\rho}(z)} a(\cdot)\lesssim \rho^\alpha$.
\end{enumerate}
Here, $\rho$ denotes the radius of the $p$-intrinsic cylinder that is defined in the stopping time argument in Section \ref{section 3}. The first and second cases correspond to the $p$- and $(p,q)$-intrinsic cylinders defined in Section \ref{section 2}. In Section \ref{section 3}, employing a stopping time argument, we study the properties regarding $p$- and $(p,q)$-intrinsic cylinders. Also, we must exclude the third case. However, under the modified assumption, it was not possible to prove this by the same method as in \cite{2023_Gradient_Higher_Integrability_for_Degenerate_Parabolic_Double-Phase_Systems}. Therefore, an additional relation between $\lambda$ and $\rho$ was needed, see Lemmas \ref{lem : no occurence with s=infty} and \ref{lem : no occurence with s<infty}. In Section \ref{section 5}, we prove the reverse H\"{o}lder inequalities for each type of intrinsic cylinder. In particular, for the $p$-intrinsic cylinder, the proof is established by dividing the cases $s=\infty$, $s\in[2,4]$ and $s\in(4,\infty)$. Finally, using the Vitali covering lemma derived in Subsection \ref{subsection 6.1} and Fubini's theorem, we complete the proof of Theorems \ref{thm : main theorem with infty} and \ref{thm : main theorem with s} in Subsection \ref{subsection 6.2}. 

\section{\bf Preliminaries}\label{section 2}
We first set up notation.
For a fixed point $z_0 \in \Omega_T$, we denote
\begin{equation}\label{def : definition of H with a fixed center z_0}
    H_{z_0}(\varkappa)\coloneq \varkappa^p+a(z_0)\varkappa^q \qquad \text{for } \varkappa\geq 0.
\end{equation}
We set intrinsic cylinders
\begin{equation}\label{eq : definition of p-intrinsic cylinder}
Q_\rho^\lambda(z_0)\coloneq  B_\rho(x_0)\times(t_0-\lambda^{2-p}\rho^2,t_0+\lambda^{2-p}\rho^2)\eqcolon B_\rho(x_0)\times I_{\rho}^\lambda(t_0)
\end{equation}
and
\begin{equation}\label{eq : definition of p,q-intrinsic cylinder}
G_\rho^\lambda(z_0)\coloneq  B_\rho(x_0)\times\left(t_0-\frac{\lambda^2}{H_{z_0}(\lambda)}\rho^2,t_0+\frac{\lambda^2}{H_{z_0}(\lambda)}\rho^2\right)\eqcolon B_\rho(x_0)\times J_\rho^\lambda(t_0).
\end{equation}
For convenience, we write 
\begin{equation*}
    \operatorname{data}=\begin{cases}
        \operatorname{data}_b &\text{if \eqref{cond : main assumption with infty} holds,}\\
        \operatorname{data}_s &\text{if \eqref{cond : main assumption with s} holds}.
    \end{cases}
\end{equation*}
Next, we denote the super-level set by
\begin{equation}\label{def : definition of Psi}
    \Psi(\Lambda)\coloneq \{z\in\Omega_T:H(z,|Du(z)|)>\Lambda\}.
\end{equation}

For any measurable set $E\subset \Omega_T$ with $0<|E|<\infty$ and integrable function $f\in L^1(\Omega_T)$, we denote the integral average of $f$ over $E$ by
$$
f_E=\frac{1}{|E|}\iints{E} f\, dz=\miint{E} f \, dz,
$$
where $|E|$ means $(n+1)$-dimensional Lebesgue measure of the set $E$.

To prove the main theorems, we refer to three energy lemmas from \cite{2023_Gradient_Higher_Integrability_for_Degenerate_Parabolic_Double-Phase_Systems}. The following lemma provides a Caccioppoli type inequality.
\begin{lemma}[\cite{2023_Gradient_Higher_Integrability_for_Degenerate_Parabolic_Double-Phase_Systems}, Lemma 2.3]\label{lem : Caccioppoli inequality}
    Let $u$ be a weak solution to \eqref{eq : main equation}. Then there exists a positive constant $c=c(n,p,q,\nu,L)$ such that
    $$
    \begin{aligned}
        &\sup_{t\in (t_0-\tau,t_0+\tau)}\dashint_{B_r(x_0)} \frac{|u-u_{Q_{r,\tau}(z_0)}|^2}{\tau}\, dx+\miint{Q_{r,\tau}(z_0)} H(z,|Du|)\, dz\\
        &\quad \leq c \miint{Q_{R, \ell}\left(z_0\right)}\left[\frac{\left|u-u_{Q_{R, \ell}\left(z_0\right)}\right|^p}{(R-r)^p}+a(z) \frac{\left|u-u_{Q_{R, \ell}\left(z_0\right)}\right|^q}{(R-r)^q}\right] dz \\
        &\qquad+c \miint{Q_{R, \ell}\left(z_0\right)} \frac{\left|u-u_{Q_{R, \ell}\left(z_0\right)}\right|^2}{\ell-\tau} \, dz
    \end{aligned}
    $$
    for every $Q_{R, \ell}\left(z_0\right)=B_R(x_0) \times(t_0-\ell, t_0+\ell) \subset \Omega_T$, with $R, \ell>0,\, r \in[R / 2, R)$ and $\tau \in [\ell / 2^2, \ell )$.
\end{lemma}
The next lemma is a gluing lemma. For this, we denote the spatial integral average of $u$ over $B_R(x_0)$ by
$$
(u)_{B_R(x_0)}(t)=\dashint_{B_R(x_0)} u(x,t)\, dx.
$$
\begin{lemma}[\cite{2023_Gradient_Higher_Integrability_for_Degenerate_Parabolic_Double-Phase_Systems}, Lemma 2.4]
    Let $u$ be a weak solution to \eqref{eq : main equation}, and let $\eta \in C_0^{\infty}(B_R(x_0))$ be a function such that
    $$
    \eta \geq 0, \quad \dashint_{B_R(x_0)} \eta \, dx=1 \quad \text { and } \quad \|\eta\|_{L^{\infty}}+R\|D \eta\|_{L^{\infty}} \leq c(n).
    $$
    Then there exists a positive constant $c=c(n, L)$ such that
\begin{multline*}
    \sup_{t_1, t_2 \in (t_0-\ell, t_0+\ell)} |(u \eta)_{B_R(x_0)}(t_2)-(u \eta)_{B_R(x_0)}(t_1)| \\
    \leq c \frac{\ell}{R} \miint{Q_{R, \ell}(z_0)}\left[|D u|^{p-1}+a(z)|D u|^{q-1}\right] dz
\end{multline*}
    for every $Q_{R, \ell}(z_0) \subset \Omega_T$ with $R, \ell>0$.
\end{lemma}
The above lemma yields a parabolic Poincar\'{e} inequality.
\begin{lemma}[\cite{2023_Gradient_Higher_Integrability_for_Degenerate_Parabolic_Double-Phase_Systems}, Lemma 2.5] \label{lem : semi-Parabolic Poincare inequality}
    Let $u$ be a weak solution to \eqref{eq : main equation}. Then there exists a positive constant $c=c(n, m, L)$ such that
\begin{multline*}
        \miint{Q_{R, \ell}(z_0)} \frac{|u-u_{\mathcal{Q}_{R, \ell}(z_0)}|^{\theta m}}{R^{\theta m}} \,dz \\
        \leq c \miint{Q_{R, \ell}(z_0)}|D u|^{\theta m} \,dz+c\left(\frac{\ell}{R^2} \iints{Q_{R, \ell}\left(z_0\right)} \left[|D u|^{p-1}+a(z)|D u|^{q-1}\right] d z\right)^{\theta m}
\end{multline*}
    for every $Q_{R, \ell}(z_0)=B_R(x_0) \times(t_0-\ell, t_0+\ell) \subset \Omega_T$ with $R, \ell>0, m \in(1, q]$ and $\theta \in(1 / m, 1]$.
\end{lemma}
\section{\bf Stopping time argument}\label{section 3}
We put
\begin{equation}\label{def : lambda_0 and Lambda_0}
\lambda_0^2\coloneq \miint{Q_{2r}(z_0)} \left[H(z,|Du|)+1\right] dz\quad \text{and}\quad \Lambda_0\coloneq \lambda_0^p+\sup_{Q_{2r}(z_0)} a(\cdot)\lambda_0^q,
\end{equation}
where $Q_{2r}(z_0)=B_{2r}(x_0)\times (t_0-(2r)^2,t_0+(2r)^2)$. Let 
\begin{equation}\label{def : definition of K and kappa}
    K\coloneq \begin{cases}
        1+80c_b[a]_\alpha &\text{if \eqref{cond : main assumption with infty} holds,}\\
        1+80c_s[a]_\alpha &\text{if \eqref{cond : main assumption with s} holds,}
    \end{cases}\quad \text{and}\quad \kappa\coloneq 10K,
\end{equation}
where $c_b$ and $c_s$ will be defined in Lemmas \ref{lem : no occurence with s=infty} and \ref{lem : no occurence with s<infty}, respectively.
For $\Psi(\Lambda)$ as in \eqref{def : definition of Psi} and $\varrho\in [r,2r]$, we write
$$
\Psi(\Lambda,\varrho)\coloneq  \Psi(\Lambda)\cap Q_\varrho(z_0)=\{z\in Q_\varrho(z_0):H(z,|Du(z)|)>\Lambda\}.
$$

We now apply a stopping time argument. Let $r\leq r_1<r_2\leq 2r$ and
$$
\Lambda>\left(\frac{4\kappa r}{r_2-r_1}\right)^{\frac{q(n+2)}{2}}\Lambda_0,
$$
where $\kappa$ is as in \eqref{def : definition of K and kappa}. For any $w\in \Psi(\Lambda,r_1)$, we choose $\lambda_w>0$ such that
\begin{equation}\label{cond : Lambda=H(lambda)}
\Lambda=\lambda_w^p+a(w)\lambda_w^q=H_w(\lambda_w),
\end{equation}
where $H_{w}$ denotes the function defined in \eqref{def : definition of H with a fixed center z_0} with $z_0$ replaced by $w$. According to \cite[Subsection 5.1]{2023_Gradient_Higher_Integrability_for_Degenerate_Parabolic_Double-Phase_Systems}, we see that there exists $\varrho_w \in (0,(r_2-r_1)/2\kappa)$ such that
\begin{equation}\label{cond : integral of H in =varrho}
\miint{Q_{\varrho_w}^{\lambda_w}(w)} H(z,|Du|)\, dz =\lambda_w^p
\end{equation}
and
\begin{equation}\label{cond : integral of H in >varrho}
\miint{Q_{\varrho}^{\lambda_w}(w)} H(z,|Du|)\, dz <\lambda_w^p
\end{equation}
for any $\varrho\in(\varrho_w,r_2-r_1)$. Moreover, we obtain from \cite[Subsection 5.1]{2023_Gradient_Higher_Integrability_for_Degenerate_Parabolic_Double-Phase_Systems} that
\begin{equation}\label{cond : relation of lambda_xi and lambda_0}
    \lambda_w\leq \left(\frac{2r}{\varrho_w}\right)^\frac{n+2}{2}\lambda_0.
\end{equation}

For $K>1$ as in \eqref{def : definition of K and kappa}, we consider the following three cases: 
\begin{enumerate}
    \item\label{case : p-phase} $\displaystyle K\lambda_w^p\geq \sup_{Q_{10\varrho_w}(w)}a(\cdot)\lambda_w^q$,
    \item\label{case : p,q-phase} $\displaystyle K\lambda_w^p\leq \sup_{Q_{10\varrho_w}(w)}a(\cdot)\lambda_w^q\qquad$ and $\qquad\displaystyle \sup_{Q_{10\varrho_w}(w)}a(\cdot)\geq 4[a]_\alpha (10\varrho_w)^\alpha$,
    \item\label{case : no occurence} $\displaystyle K\lambda_w^p\leq \sup_{Q_{10\varrho_w}(w)}a(\cdot)\lambda_w^q\qquad$ and $\qquad\displaystyle \sup_{Q_{10\varrho_w}(w)}a(\cdot)\leq 4[a]_\alpha (10\varrho_w)^\alpha$.
\end{enumerate}
\textbf{Case \eqref{case : p-phase}}: By using \eqref{cond : integral of H in =varrho} and \eqref{cond : integral of H in >varrho} and replacing the center point $w$, radius $\varrho_w$ and $\lambda_w$ with $z_0$, $\rho$ and $\lambda$, respectively, we obtain
\begin{equation}\label{cond : p-phase condition}
    \left\{\begin{aligned}
        &K\lambda^p\geq \sup_{Q_{10\rho}(z_0)} a(\cdot)\lambda^q,\\
        &\miint{Q_\sigma^\lambda (z_0)} H(z,|Du|)\, dz <\lambda^p \quad \text{for any }\sigma\in(\rho,2\kappa\rho],\\
        &\miint{Q_\rho^\lambda(z_0)} H(z,|Du|)\, dz = \lambda^p.
    \end{aligned}\right.
\end{equation}
\textbf{Case \eqref{case : p,q-phase}}: We obtain from \eqref{case : p,q-phase}$_2$ that
$$
4[a]_\alpha (10\varrho_w)^\alpha\leq \sup_{Q_{10\varrho_w}(w)} a(\cdot)\leq \inf_{Q_{10\varrho_w}(w)}a(\cdot)+2[a]_\alpha (10\varrho_w)^\alpha,
$$
and hence
$$
\sup_{Q_{10\varrho_w}(w)}a(\cdot)\leq \inf_{Q_{10\varrho_w} (w)} a(\cdot) +2[a]_\alpha (10\varrho_w)^\alpha\leq 2\inf_{Q_{10\varrho_w} (w)} a(\cdot).
$$
Therefore, we get
\begin{equation}    \label{cond : comparison in p,q-phase}
    \frac{a(w)}{2}\leq a(\tilde{w})\leq 2a(w) \quad \text{for every } \tilde{w}\in Q_{10\varrho_w} (w).
\end{equation}
It follows from \eqref{case : p,q-phase}$_2$ and \eqref{cond : comparison in p,q-phase} that $a(w)>0$ and $G_\sigma^{\lambda_w}(w) \varsubsetneq Q_\sigma^{\lambda_w}(w)$. By \cite[Subsection 5.1]{2023_Gradient_Higher_Integrability_for_Degenerate_Parabolic_Double-Phase_Systems}, there exists $\varsigma_w\in (0,\varrho_w]$ such that
\begin{equation}\label{cond : integral of H in =varsigma in p,q-phase}
    \miint{G^{\lambda_w}_{\varsigma_w}(w)} H(z,|Du|)\, dz=H_w(\lambda_w)
\end{equation}
and
\begin{equation}\label{cond : integral of H in >varsigma in p,q-phase}
    \miint{G^{\lambda_w}_{\sigma}(w)} H(z,|Du|)\, dz<H_w(\lambda_w)
\end{equation}
for any $\sigma\in(\varsigma_w,r_2-r_1)$. Hence, if we replace the center point $w$, radius $\varsigma_w$ and $\lambda_w$ in \eqref{cond : comparison in p,q-phase}-\eqref{cond : integral of H in >varsigma in p,q-phase} with $z_0$, $\rho$ and $\lambda$, respectively, we obtain
\begin{equation}\label{cond : p,q-phase condition}
    \left\{\begin{aligned}
        &K\lambda^p\leq \sup_{Q_{10\rho}(z_0)} a(\cdot)\lambda^q,\quad \frac{a(z_0)}{2}\leq a(z)\leq 2a(z_0)\quad \text{for every }z\in G_{4\rho}^\lambda(z_0),\\
        &\miint{G_\sigma^\lambda (z_0)} H(z,|Du|)\, dz <H_{z_0}(\lambda) \quad \text{for any }\sigma\in(\rho,2\kappa\rho],\\
        &\miint{G_\rho^\lambda(z_0)} H(z,|Du|)\, dz = H_{z_0}(\lambda).
    \end{aligned}\right.
\end{equation}
\textbf{Case \eqref{case : no occurence}}: We shall rigorously exclude the possibility of this case by proving the estimates
\begin{equation}\label{cond : the impossibility of Case (3)}
\left\{\begin{aligned}
    &\lambda_w \lesssim \varrho_w^{-1} \quad\quad \text{if } \eqref{cond : main assumption with infty} \text{ holds},\\
    &\lambda_w \lesssim \varrho_w^{-\frac{n+s}{s}} \quad \text{if } \eqref{cond : main assumption with s} \text{ holds}.
\end{aligned}\right.
\end{equation}
\begin{lemma}\label{lem : no occurence with s=infty}
    Let $u$ be a weak solution to \eqref{eq : main equation}, and suppose that
\begin{equation}\label{cond : no occurence}
    \sup_{Q_{10\varrho_w}(w)}a(\cdot)\leq 4[a]_\alpha (10\varrho_w)^\alpha.
\end{equation}
    If \eqref{cond : main assumption with infty} holds, then there exists a constant $c_b=c_b(\operatorname{data}_b)>1$ such that
    $$
    \varrho_w \leq c_b\lambda_w^{-1}.
    $$
\end{lemma}
\begin{proof}
    By Lemma \ref{lem : Caccioppoli inequality} and \eqref{cond : integral of H in =varrho}, we get
    \begin{align}
    \lambda_w^p&=\miint{Q_{\varrho_w}^{\lambda_w}(w)}H(z,|Du|)\,dz \nonumber\\
    &\leq c\miint{Q_{2\varrho_w}^{\lambda_w}(w)}\left[\frac{\Big|u-u_{Q_{2\varrho_w}^{\lambda_w}(w)}\Big|^p}{(2\varrho_w)^p}+a(z)\frac{\Big|u-u_{Q_{2\varrho_w}^{\lambda_w}(w)}\Big|^q}{(2\varrho_w)^q}\right] dz \nonumber\\ 
    &\qquad +c\lambda_w^{p-2}\miint{Q_{2\varrho_w}^{\lambda_w}(w)}\frac{\Big|u-u_{Q_{2\varrho_w}^{\lambda_w}(w)}\Big|^2}{(2\varrho_w)^2}\,dz \nonumber\\ \label{eq : Caccioppoli ineq in no occurence with infty}
    &=\mathrm{I}_1+\mathrm{I}_2+\mathrm{I}_3
    \end{align}
    for some $c=c(n,p,q,\nu,L)>1$. We note from the triangle inequality and Jensen's inequality that
    \begin{equation}\label{eq : estimation of u-mean of u as u}
        \miint{Q_{2\varrho_w}^{\lambda_w}(w)}\frac{\Big|u-u_{Q_{2\varrho_w}^{\lambda_w}(w)}\Big|^\gamma}{(2\varrho_w)^\gamma}\,dz\leq c(\gamma)\miint{Q_{2\varrho_w}^{\lambda_w}(w)}\frac{|u|^\gamma}{(2\varrho_w)^\gamma} \, dz
    \end{equation}
    holds for any $\gamma\in [2,\infty)$.

    \textbf{Estimate of $\mathrm{I}_1$.} By \eqref{eq : estimation of u-mean of u as u}, we get
    $$
    \mathrm{I}_1\leq c\miint{Q_{2\varrho_w}^{\lambda_w}(w)}\frac{|u|^p}{(2\varrho_w)^p} \, dz\leq c\varrho_w^{-p}
    $$
    for some $c=c(n,p,q,\nu,L,\|u\|_{L^\infty(\Omega_T)})>1$.

    \textbf{Estimate of $\mathrm{I}_2$.} By \eqref{cond : no occurence}, \eqref{eq : estimation of u-mean of u as u} and \eqref{cond : main assumption with infty}$_3$, we get
    $$
    \mathrm{I}_2\leq c\varrho_w^\alpha\miint{Q_{2\varrho_w}^{\lambda_w}(w)}\frac{|u|^q}{(2\varrho_w)^q} \, dz\leq c\|u\|_{L^\infty(\Omega_T)}^q \varrho_w^{\alpha-q}\leq c\varrho_w^{-p}
    $$
    for some $c=c(n,p,q,\alpha,\nu,L,\operatorname{diam}(\Omega),[a]_\alpha,\|u\|_{L^\infty(\Omega_T)})>1$.

    \textbf{Estimate of $\mathrm{I}_3$.} By \eqref{eq : estimation of u-mean of u as u} and Young's inequality, we get
    $$
    \mathrm{I}_3\leq c\lambda_w^{p-2}\miint{Q_{2\varrho_w}^{\lambda_w}(w)}\frac{|u|^2}{(2\varrho_w)^2} \, dz\leq c\|u\|_{L^\infty(\Omega_T)}^2 \lambda_w^{p-2} \varrho_w^{-2}\leq \frac{1}{2}\lambda_w^p+c\varrho_w^{-p}
    $$
    for some $c=c(n,p,q,\nu,L,\|u\|_{L^\infty(\Omega_T)})>1$.

    Combining the above results with \eqref{eq : Caccioppoli ineq in no occurence with infty}, we conclude that
    $$
    \lambda_w^p\leq c\varrho_w^{-p}.
    $$
\end{proof}
Next, we prove \eqref{cond : the impossibility of Case (3)}$_2$ using the Gagliardo-Nirenberg multiplicative embedding inequality.
\begin{lemma}\label{lem : no occurence with s<infty}
Let $u$ be a weak solution to \eqref{eq : main equation}, and suppose that \eqref{cond : no occurence} is satisfied.
If \eqref{cond : main assumption with s} holds for some $2\leq s <\infty$, then there exists a constant $c_s=c_s(\operatorname{data}_s)>1$ such that
$$
\varrho_w \leq c_s \lambda_w^{-\frac{s}{n+s}}.
$$
\end{lemma}
\begin{proof}
    As in Lemma \ref{lem : no occurence with s=infty}, we deduce that Lemma \ref{lem : Caccioppoli inequality} and \eqref{cond : integral of H in =varrho} imply
    \begin{align}
        \lambda_w^p&=\miint{Q_{\varrho_w}^{\lambda_w}(w)}H(z,|Du|)\,dz \nonumber\\
        &\leq c\miint{Q_{2\varrho_w}^{\lambda_w}(w)}\left[\frac{\Big|u-u_{Q_{2\varrho_w}^{\lambda_w}(w)}\Big|^p}{(2\varrho_w)^p}+a(z)\frac{\Big|u-u_{Q_{2\varrho_w}^{\lambda_w}(w)}\Big|^q}{(2\varrho_w)^q}\right] dz \nonumber\\ \label{eq : Caccioppoli ineq in no occurence}
        &\quad +c\lambda_w^{p-2}\miint{Q_{2\varrho_w}^{\lambda_w}(w)}\frac{\Big|u-u_{Q_{2\varrho_w}^{\lambda_w}(w)}\Big|^2}{(2\varrho_w)^2}\,dz=\mathrm{I}_1+\mathrm{I}_2+\mathrm{I}_3
    \end{align}
    for some $c=c(n,p,q,\nu,L)>1$.    
    
    \textbf{Estimate of $\mathrm{I}_3$.} We obtain from \eqref{eq : estimation of u-mean of u as u}, H\"{o}lder's inequality and Young's inequality that
    \begin{align}
        \mathrm{I}_3&\leq c\lambda_w^{p-2}\miint{Q_{2\varrho_w}^{\lambda_w}(w)}\frac{|u|^2}{(2\varrho_w)^2} \, dz \nonumber\\
        &\leq c\lambda_w^{p-2}\left(\miint{Q_{2\varrho_w}^{\lambda_w}(w)}\frac{|u|^s}{(2\varrho_w)^s} \, dz\right)^{\frac{2}{s}} \nonumber\\
        &\leq c\frac{\lambda_w^{p-2}}{(2\varrho_w)^2}\left(\frac{1}{|B_{2\varrho_w}|}\sup_{t\in [0,T]}\int_{\Omega}|u|^s\,dx\right)^\frac{2}{s} \nonumber\\
        &\leq c\lambda_w^{p-2}\varrho_w^{-\frac{2(n+s)}{s}} \nonumber\\ \label{eq : estimate of I_3 in lem_no occurence}
        &\leq \frac{1}{4}\lambda_w^p+c\varrho_w^{-\frac{p(n+s)}{s}} 
    \end{align}
    for some $c=c(n,p,q,s,\nu,L,\|u\|_{L^\infty(0,T;L^s(\Omega))})>1$.

    \textbf{Estimate of $\mathrm{I}_1$.} To estimate $\mathrm{I}_1$, we divide the cases according to $p$ and $s$. 
    
    First, if $p\leq s$, by \eqref{eq : estimation of u-mean of u as u} and H\"{o}lder's inequality, we get 
    \begin{align}
        \mathrm{I}_1&\leq c\miint{Q_{2\varrho_w}^{\lambda_w}(w)}\frac{|u|^p}{(2\varrho_w)^p} \, dz \nonumber\\
        &\leq c\left(\miint{Q_{2\varrho_w}^{\lambda_w}(w)}\frac{|u|^s}{(2\varrho_w)^s} \, dz\right)^{\frac{p}{s}} \nonumber\\ \label{eq : estimate of I_1 in lem_no occurence with p<s}
        &\leq c\varrho_w^{-\frac{p(n+s)}{s}}
    \end{align}
    for some $c=c(n,p,q,s,\nu,L,\|u\|_{L^\infty(0,T;L^s(\Omega))})>1$.

    Now, suppose that $p>s$. Then we obtain
    \begin{align*}
        \mathrm{I}_1&\leq c\miint{Q_{2\varrho_w}^{\lambda_w}(w)}\frac{\Big|u-u_{B_{2\varrho_w}(x_0)}(t)\Big|^p}{(2\varrho_w)^p}\,dx\,dt+c\dashint_{I_{2\varrho_w}^{\lambda_w}(t_0)}\frac{\Big|u_{B_{2\varrho_w}(x_0)}(t)-u_{Q_{2\varrho_w}^{\lambda_w}(w)}\Big|^p}{(2\varrho_w)^p}\, dt \\ 
        &=\mathrm{J}_1+\mathrm{J}_2,
    \end{align*}
    where $w=(x_0, t_0)$. By the Gagliardo-Nirenberg multiplicative embedding inequality in \cite[Theorem 2.1 and Remark 2.1]{DiBenedetto1983}, we have
    $$
    \begin{aligned}
        &\mathrm{J}_1= c\dashint_{I_{2\varrho_w}^{\lambda_w}(t_0)}\left(\dashint_{B_{2\varrho_w}(x_0)}\frac{\Big|u-u_{B_{2\varrho_w}(x_0)}(t)\Big|^p}{(2\varrho_w)^p}\,dx\right)\,dt\\
        &\;\;\leq c\dashint_{I_{2\varrho_w}^{\lambda_w}(t_0)}\left(\dashint_{B_{2\varrho_w}(x_0)} |Du|^p\, dx\right)^{\theta_1}\left(\dashint_{B_{2\varrho_w}(x_0)}\frac{\Big|u-u_{B_{2\varrho_w}(x_0)}(t)\Big|^s}{(2\varrho_w)^s}\,dx\right)^{\frac{p(1-\theta_1)}{s}}\,dt,
    \end{aligned}
    $$
    where $\theta_1=\left(\frac{1}{s}-\frac{1}{p}\right)\left(\frac{1}{n}+\frac{1}{s}-\frac{1}{p}\right)^{-1}\in [0,1]$ and $c=c(n,p,q,s,\nu,L)>1$. If follows from \eqref{cond : integral of H in >varrho}, \eqref{eq : estimation of u-mean of u as u}, and Young's inequality that
    \begin{align*}
        \mathrm{J}_1&\leq c\|u\|_{L^\infty(0,T;L^s(\Omega))}\varrho_w^{-\frac{p(n+s)(1-\theta_1)}{s}}\lambda_w^{p\theta_1} \\
        &\leq \frac{1}{4}\lambda_w^p+c\varrho_w^{-\frac{p(n+s)}{s}}.
    \end{align*}
    Here, $c>1$ depends on $n,p,q,s,\nu,L$ and $\|u\|_{L^\infty(0,T;L^s(\Omega))}$.
    Next, $\mathrm{J}_2$ is easily calculated as follows:
    \begin{align*}
        \mathrm{J}_2&\leq c\varrho_w^{-p}\dashint_{I_{2\varrho_w}^{\lambda_w}(t_0)}\dashint_{I_{2\varrho_w}^{\lambda_w}(t_0)} |u_{B_{2\varrho_w}(x_0)}(\tilde{t})-u_{B_{2\varrho_w}(x_0)}(t)|^p\,dt\,d\tilde{t} \\
        &\leq c\varrho_w^{-p}\dashint_{I_{2\varrho_w}^{\lambda_w}(t_0)} |u_{B_{2\varrho_w}(x_0)}(t)|^p \, dt \\
        &\leq c\varrho_w^{-p}\sup_{I_{2\varrho_w}^{\lambda_w}(t_0)}\left(\dashint_{B_{2\varrho_w}(x_0)}|u|\,dx\right)^p \\
        &\leq c\varrho_w^{-p}\sup_{I_{2\varrho_w}^{\lambda_w}(t_0)}\left(\dashint_{B_{2\varrho_w}(x_0)}|u|^s\,dx\right)^\frac{p}{s} \\
        &\leq c\varrho_w^{-\frac{p(n+s)}{s}}\sup_{I_{2\varrho_w}^{\lambda_w}(t_0)}\left(\int_{\Omega}|u|^s\,dx\right)^\frac{p}{s}\\ 
        &\leq c\varrho_w^{-\frac{p(n+s)}{s}}
    \end{align*}
    for some $c=c(n,p,q,s,\nu,L,\|u\|_{L^\infty(0,T;L^s(\Omega))})>1$. Hence, we obtain
    \begin{equation}\label{eq : estimate of I_1 in lem_no occurence with p>s}
        \mathrm{I}_1\leq \frac{1}{4}\lambda_w^p+c\varrho_w^{-\frac{p(n+s)}{s}}.
    \end{equation}
    We conclude from \eqref{eq : estimate of I_1 in lem_no occurence with p<s} and \eqref{eq : estimate of I_1 in lem_no occurence with p>s} that
    \begin{equation}\label{eq : estimate of I_1 in lem_no occurence}
        \mathrm{I}_1\leq \frac{1}{4}\lambda_w^p+c\varrho_w^{-\frac{p(n+s)}{s}}
    \end{equation}
    for some $c=c(n,p,q,s,\nu,L,\|u\|_{L^\infty(0,T;L^s(\Omega))})>1$.

    \textbf{Estimate of $\mathrm{I}_2$.} By \eqref{cond : no occurence}, we have
    $$
    \mathrm{I}_2\leq c\varrho_w^\alpha \miint{Q_{2\varrho_w}^{\lambda_w}(w)}\frac{\Big|u-u_{Q_{2\varrho_w}^{\lambda_w}(w)}\Big|^q}{(2\varrho_w)^q}\,dz
    $$
    for some $c=c(n,p,q,\alpha,\nu,L,[a]_\alpha)>1$. We divide the cases according to $q$ and $s$ in the same way as we estimated $\mathrm{I}_1$.

    If $q\leq s$, since \eqref{cond : main assumption with s}$_3$ implies that $-\frac{p(n+s)}{s}\leq \alpha -\frac{q(n+s)}{s}<0$, it follows from \eqref{eq : estimation of u-mean of u as u} and H\"{o}lder's inequality that
    \begin{align}
        \mathrm{I}_2&\leq c\varrho_w^\alpha \left(\miint{Q_{2\varrho_w}^{\lambda_w}(w)}\frac{\Big|u-u_{Q_{2\varrho_w}^{\lambda_w}(w)}\Big|^s}{(2\varrho_w)^s}\,dz\right)^{\frac{q}{s}}\nonumber\\ \label{eq : estimate of I_2 in lem_no occurence with q=s}
        &\leq c \varrho_w^{\alpha-\frac{q(n+s)}{s}}\leq c\varrho_w^{-\frac{p(n+s)}{s}}
    \end{align}
    for some $c=c(n,p,q,s,\alpha,\nu,L,\operatorname{diam}(\Omega),[a]_\alpha,\|u\|_{L^\infty(0,T;L^s(\Omega))})>1$.

    Finally, assume that $q>s$. Then we obtain
    \begin{align*}
        \mathrm{I}_2&\leq c\varrho_w^\alpha\miint{Q_{2\varrho_w}^{\lambda_w}(w)}\frac{\Big|u-u_{B_{2\varrho_w}(x_0)}(t)\Big|^q}{(2\varrho_w)^q}\,dx dt\\
        &\qquad+c\varrho_w^\alpha\dashint_{I_{2\varrho_w}^{\lambda_w}(t_0)}\frac{\Big|u_{B_{2\varrho_w}(x_0)}(t)-u_{Q_{2\varrho_w}^{\lambda_w}(w)}\Big|^q}{(2\varrho_w)^q}\, dt \\ 
        &=\mathrm{J}_3+\mathrm{J}_4,
    \end{align*}
    where $w=(x_0, t_0)$. By the Gagliardo-Nirenberg multiplicative embedding inequality in \cite[Theorem 2.1 and Remark 2.1 in Section I]{1993_Degenerate_parabolic_equations_DiBenedetto}, we get
    $$
    \begin{aligned}
        \mathrm{J}_3&= c\varrho_w^\alpha\dashint_{I_{2\varrho_w}^{\lambda_w}(t_0)}\left(\dashint_{B_{2\varrho_w}(x_0)}\frac{\Big|u-u_{B_{2\varrho_w}(x_0)}(t)\Big|^q}{(2\varrho_w)^q}\,dx\right)\,dt\\
        &\leq c\varrho_w^\alpha\dashint_{I_{2\varrho_w}^{\lambda_w}(t_0)}\left(\dashint_{B_{2\varrho_w}(x_0)} |Du|^p\, dx\right)^{\frac{q\theta_2}{p}}\\
        &\qquad\qquad\qquad \times \left(\dashint_{B_{2\varrho_w}(x_0)}\frac{\Big|u-u_{B_{2\varrho_w}(x_0)}(t)\Big|^s}{(2\varrho_w)^s}\,dx\right)^{\frac{q(1-\theta_2)}{s}}\,dt,
    \end{aligned}
    $$
    where $\theta_2=\left(\frac{1}{s}-\frac{1}{q}\right)\left(\frac{1}{n}+\frac{1}{s}-\frac{1}{p}\right)^{-1}$ and $c=c(n,p,q,s,\nu,L)>1$. We shall check that $\theta_2$ is in $[0,1]$. For this, we need
    \begin{equation}\label{cond : relation of p, q, n}
        \frac{1}{q}+\frac{1}{n}-\frac{1}{p}\geq 0.
    \end{equation}
    Indeed, since $q\leq p+\frac{s\alpha}{n+s}\leq p+\frac{s}{n+s}$, if $\frac{1}{p+\frac{s}{n+s}}\geq\frac{n-p}{np}$ holds, then \eqref{cond : relation of p, q, n} holds. Moreover, a simple calculation implies that
    \begin{align}
        \frac{1}{p+\frac{s}{n+s}}\geq\frac{n-p}{np} \quad&\iff \quad np\geq (n-p)\left(p+\frac{s}{n+s}\right) \nonumber\\ \label{cond : relation of p,q,n,s}
        &\iff\quad p^2 n-sn+p^2 s+ps \geq 0,
    \end{align}
    and since $s<q\leq p+1\leq p^2$, \eqref{cond : relation of p,q,n,s} holds. Hence \eqref{cond : relation of p, q, n} is satisfied. Now, if $q> p\geq s$, then clearly $\theta_2\geq 0$. If $q> s\geq p$, then we get from \eqref{cond : relation of p, q, n} that
    $$
    \frac{1}{s}+\frac{1}{n}-\frac{1}{p}> \frac{1}{q}+\frac{1}{n}-\frac{1}{p}\geq 0, 
    $$
    and hence $\theta_2\geq 0$. Finally, \eqref{cond : relation of p, q, n} implies $\theta_2\leq 1$. Next, to use H\"{o}lder's inequality and Young's inequality, we shall check $\frac{p}{q\theta_2}> 1.$
    Since \begin{align*}
        \frac{p}{q\theta_2}&=\frac{p}{q}\bigg(\frac{1}{s}+\frac{1}{n}-\frac{1}{p}\bigg)\bigg(\frac{1}{s}-\frac{1}{q}\bigg)^{-1}\\
        &=\bigg(\frac{p}{s}+\frac{p}{n}-1\bigg)\bigg(\frac{q}{s}-1\bigg)^{-1},
    \end{align*}
    we only need to show
    $$
    \frac{p}{s}+\frac{p}{n}-1> \frac{q}{s}-1.
    $$
    Since $q\leq p+\frac{s}{n+s}$, $2\leq p$ and $n< n+s$, we see that $\frac{q}{s}-1\leq \frac{p}{s}+\frac{1}{n+s}-1< \frac{p}{s}+\frac{p}{n}-1$. Thus, we obtain $\frac{p}{q\theta_2}> 1.$
    Then it follows from H\"{o}lder inequality and Young's inequality that
    \begin{align*}
        \mathrm{J}_3&\leq c\varrho_w^{\alpha-\frac{q(n+s)(1-\theta_2)}{s}}\left(\miint{Q_{2\varrho_w}^{\lambda_w}(w)} |Du|^p\, dz\right)^{\frac{q\theta_2}{p}} \\
        &\leq c\varrho_w^{\alpha-\frac{q(n+s)(1-\theta_2)}{s}}\lambda_w^{q\theta_2} \\
        &\leq \frac{1}{4}\lambda_w^{p}+c\varrho_w^{\left(\alpha-\frac{q(n+s)(1-\theta_2)}{s}\right)\cdot\frac{p}{p-q\theta_2}}.
    \end{align*}
    Since $\alpha\geq \frac{(q-p)(n+s)}{s}$, we observe that
    $$
    \begin{aligned}
        \left(\alpha-\frac{q(n+s)(1-\theta_2)}{s}\right)\cdot\frac{p}{p-q\theta_2}&\geq \left(q-p-q(1-\theta_2)\right)\cdot\frac{1}{p-q\theta_2}\cdot \frac{p(n+s)}{s}\\
        &=-\frac{p(n+s)}{s}.
    \end{aligned}
    $$
    We claim that $\alpha-\frac{q(n+s)(1-\theta_2)}{s}<0$. Indeed, since $\theta_2=\left(\frac{1}{s}-\frac{1}{q}\right)\left(\frac{1}{n}+\frac{1}{s}-\frac{1}{p}\right)^{-1}$, a simple calculation yields
    \begin{align*}
        &\alpha-\frac{q(n+s)(1-\theta_2)}{s}<0 \\
        &\qquad\qquad\iff\quad \alpha-\frac{q(n+s)}{s}<-\frac{q(n+s)}{s}\theta_2\\
        &\qquad\qquad\iff\quad \left(\alpha-\frac{q(n+s)}{s}\right)\left(\frac{1}{n}+\frac{1}{s}-\frac{1}{p}\right)<-\frac{q(n+s)}{s}\left(\frac{1}{s}-\frac{1}{q}\right)\\
        &\qquad\qquad\iff\quad \frac{\alpha}{n}+\frac{\alpha}{s}+\frac{q(n+s)}{ps}<\frac{\alpha}{p}+\frac{q(n+s)}{ns}+\frac{n+s}{s}.
    \end{align*}
    Since $\alpha\in (0,1)$ and $2\leq q$, we obtain
    $\frac{\alpha}{n}<\frac{q}{n}$, $\frac{\alpha}{s}<\frac{q}{s}$ and $\frac{q(n+s)}{ps}\leq \frac{n+s}{s}+\frac{\alpha}{p}$ by \eqref{cond : main assumption with s}. Therefore, the claim holds. Then we get
    \begin{equation*}
        \mathrm{J}_3\leq \frac{1}{4}\lambda_w^{p}+c\varrho_w^{-\frac{p(n+s)}{s}}
    \end{equation*}
    for some $c=c(n,p,q,s,\alpha,\nu,L,\operatorname{diam}(\Omega),[a]_\alpha,\|u\|_{L^\infty(0,T;L^s(\Omega))})>1$.
    As in $\mathrm{J}_2$, we have
    $$
    \mathrm{J}_4\leq c\varrho_w^{\alpha-\frac{q(n+s)}{s}}\leq c\varrho_w^{-\frac{p(n+s)}{s}}.
    $$
    Thus, we obtain
    \begin{equation}\label{eq : estimate of I_2 in lem_no occurence with q>s}
        \mathrm{I}_2\leq \frac{1}{4}\lambda_w^p +c \varrho_w^{-\frac{p(n+s)}{s}}.
    \end{equation}
    By \eqref{eq : estimate of I_2 in lem_no occurence with q=s} and \eqref{eq : estimate of I_2 in lem_no occurence with q>s}, we conclude that
    \begin{equation}\label{eq : estimate of I_2 in lem_no occurence}
        \mathrm{I}_2\leq \frac{1}{4}\lambda_w^p+c\varrho_w^{-\frac{p(n+s)}{s}}
    \end{equation}
    for some $c=c(n,p,q,s,\alpha,\nu,L,\operatorname{diam}(\Omega),[a]_\alpha,\|u\|_{L^\infty(0,T;L^s(\Omega))})>1$.

    Combining \eqref{eq : Caccioppoli ineq in no occurence}, \eqref{eq : estimate of I_3 in lem_no occurence}, \eqref{eq : estimate of I_1 in lem_no occurence} and \eqref{eq : estimate of I_2 in lem_no occurence} gives
    $$
    \lambda_w^p\leq c\varrho_w^{-\frac{p(n+s)}{s}},
    $$
    which completes the proof.
\end{proof}

Now, we show that the case \eqref{case : no occurence} never occurs. If \eqref{case : no occurence} holds, we have
$$
K\lambda_w^p=\sup_{Q_{10\varrho_w}(w)} a(\cdot)\frac{K\lambda_w^p}{\displaystyle\sup_{Q_{10\varrho_w}(w)} a(\cdot)}\leq 40[a]_\alpha\varrho_w^\alpha \lambda_w^q.
$$
When \eqref{cond : main assumption with infty} holds, then it follows from Lemma \ref{lem : no occurence with s=infty} and \eqref{def : definition of K and kappa} that
$$
K\lambda_w^p\leq 40[a]_\alpha\varrho_w^\alpha \lambda_w^q\leq 40c_b[a]_\alpha\lambda_w^{q-\alpha}\leq 40c_b[a]_\alpha\lambda_w^{p}<\frac{K}{2}\lambda_w^p,
$$
which is a contradiction. Similarly, when \eqref{cond : main assumption with s} holds, then it follows from Lemma \ref{lem : no occurence with s<infty} and \eqref{def : definition of K and kappa} that
$$
K\lambda_w^p\leq 40[a]_\alpha\varrho_w^\alpha \lambda_w^q\leq 40c_s[a]_\alpha\lambda_w^{q-\frac{s\alpha}{n+s}}\leq 40c_s[a]_\alpha\lambda_w^{p}<\frac{K}{2}\lambda_w^p,
$$
which is a contradiction. Thus, the case \eqref{case : no occurence} can never happen under either \eqref{cond : main assumption with infty} or \eqref{cond : main assumption with s}.

\section{\bf Parabolic Sobolev-Poincar\'{e} type inequalities}\label{section 4}
Let $z_0=(x_0,t_0)\in\Psi(\Lambda)$ be a Lebesgue point of $|Du(z)|^p+a(z)|Du(z)|^q$, where $\Lambda$ is defined in Section \ref{section 3}. To prove the parabolic Sobolev-Poincar\'{e} inequality, we divide it into two cases: $p$-phase and $(p,q)$-phase.
\subsection{\bf The $p$-phase case}
We assume \eqref{cond : p-phase condition} and estimate the last term in Lemma \ref{lem : semi-Parabolic Poincare inequality}.
\begin{lemma}   \label{lem : last term estimate in Lemma lem : semi-Parabolic Poincare inequality whenever p-intrinsic cylinder}
    Let $u$ be a weak solution to \eqref{eq : main equation} and assume that $Q_{4\rho}^\lambda(z_0)\subset \Omega_T$ satisfies \eqref{cond : p-phase condition}. Then, for $\sigma\in[2\rho,4\rho]$, there exists a positive constant $c=c(\operatorname{data})$ such that
    $$
    \begin{aligned}
        \miint{Q_\sigma^\lambda(z_0)} \left[|Du|^{p-1}+a(z)|Du|^{q-1}\right] dz&\leq \miint{Q_\sigma^\lambda(z_0)} |Du|^{p-1}\, dz\\
        &\quad +c\lambda^{-1+\frac{p}{q}}\miint{Q_\sigma^\lambda(z_0)} a(z)^{\frac{q-1}{q}}|Du|^{q-1}\, dz.\\
    \end{aligned}
    $$
\end{lemma}
\begin{proof}
    By \eqref{cond : p-phase condition}$_1$, there exists a positive constant $c=c(K)$ such that
    \begin{align*}
        &\miint{Q_\sigma^\lambda(z_0)} \left[|Du|^{p-1}+a(z)|Du|^{q-1}\right] dz\\
        &\quad\quad\leq \miint{Q_\sigma^\lambda(z_0)} |Du|^{p-1}\, dz+\sup_{Q_{10\rho}(z_0)} a(\cdot)^{\frac{1}{q}}\miint{Q_\sigma^\lambda(z_0)} a(z)^{\frac{q-1}{q}}|Du|^{q-1}\, dz\\
        &\quad\quad\leq \miint{Q_\sigma^\lambda(z_0)} |Du|^{p-1}\, dz+ c\lambda^{-1+\frac{p}{q}}\miint{Q_\sigma^\lambda(z_0)} a(z)^{\frac{q-1}{q}}|Du|^{q-1}\, dz.
    \end{align*}
\end{proof}
Next, we establish a $p$-intrinsic parabolic Poincar\'{e} inequality.
\begin{lemma}\label{lem : p-intrinsic parabolic Poincare inequality of p-term in p-intrinsic cylinder}
    Let $u$ be a weak solution to \eqref{eq : main equation} and assume that $Q_{4\rho}^\lambda(z_0)\subset \Omega_T$ satisfies \eqref{cond : p-phase condition}. Then, for $\sigma\in[2\rho,4\rho]$ and $\theta \in (\frac{q-1}{p},1]$, there exists a positive constant $c=c(\operatorname{data})$ such that
    \begin{align*}
        \miint{Q_\sigma^\lambda(z_0)}\frac{\Big|u-u_{Q_\sigma^\lambda(z_0)}\Big|^{\theta p}}{\sigma^{\theta p}}\, dz\leq c\miint{Q_\sigma^\lambda(z_0)} H(z,|Du|)^\theta\, dz.
    \end{align*}
\end{lemma}
\begin{proof}
    By Lemmas \ref{lem : semi-Parabolic Poincare inequality} and \ref{lem : last term estimate in Lemma lem : semi-Parabolic Poincare inequality whenever p-intrinsic cylinder}, there exists a positive constant $c=c(\operatorname{data})$ such that 
    \begin{align*}
        \miint{Q_\sigma^\lambda(z_0)}\frac{\Big|u-u_{Q_\sigma^\lambda(z_0)}\Big|^{\theta p}}{\sigma^{\theta p}}\, dz &\leq  c \miint{Q_\sigma^\lambda(z_0)}|D u|^{\theta p} \,dz\\
        &\quad +c\left(\lambda^{2-p}\miint{Q_\sigma^\lambda(z_0)} |Du|^{p-1}\, dz\right)^{\theta p}\\
        &\quad +c\left(\lambda^{1-p+\frac{p}{q}}\miint{Q_\sigma^\lambda(z_0)} a(z)^{\frac{q-1}{q}}|Du|^{q-1}\, dz\right)^{\theta p}.
    \end{align*}
    We conclude from H\"{o}lder's inequality and \eqref{cond : p-phase condition} that
    \begin{align*}
        &\miint{Q_\sigma^\lambda(z_0)}\frac{\Big|u-u_{Q_\sigma^\lambda(z_0)}\Big|^{\theta p}}{\sigma^{\theta p}}\, dz\\ 
        &\quad\leq  c \miint{Q_\sigma^\lambda(z_0)}|D u|^{\theta p} \,dz\\
        &\qquad +c\lambda^{(2-p)\theta p}\left(\miint{Q_\sigma^\lambda(z_0)} |Du|^{\theta p}\, dz\right)^{p-1}\\
        &\qquad +c\lambda^{\left(1-p+\frac{p}{q}\right)\theta p}\left(\miint{Q_\sigma^\lambda(z_0)} a(z)^\theta|Du|^{\theta q}\, dz\right)^{p-\frac{p}{q}}\\
        &\quad\leq  c \miint{Q_\sigma^\lambda(z_0)}|D u|^{\theta p} \,dz\\
        &\qquad +c\lambda^{(2-p)\theta p}\left(\miint{Q_\sigma^\lambda(z_0)} |Du|^{\theta p}\, dz\right)\left(\miint{Q_\sigma^\lambda(z_0)} |Du|^{p}\, dz\right)^{(p-2)\theta}\\
        &\qquad +c\lambda^{\left(1-p+\frac{p}{q}\right)\theta p}\left(\miint{Q_\sigma^\lambda(z_0)} a(z)^\theta|Du|^{\theta q}\, dz\right)\left(\miint{Q_\sigma^\lambda(z_0)} a(z)|Du|^{q}\, dz\right)^{\left(-1+p-\frac{p}{q}\right)\theta}\\
        &\quad\leq  c \miint{Q_\sigma^\lambda(z_0)}H(z,|Du|)^{\theta} \,dz.
    \end{align*}
\end{proof}
\begin{lemma}\label{lem : p-intrinsic parabolic Poincare inequality of q-term in p-intrinsic cylinder}
    Let $u$ be a weak solution to \eqref{eq : main equation} and assume that $Q_{4\rho}^\lambda(z_0)\subset \Omega_T$ satisfies \eqref{cond : p-phase condition}. Then, for $\sigma\in[2\rho,4\rho]$ and $\theta \in (\frac{q-1}{p},1]$, there exists a positive constant $c=c(\operatorname{data})$ such that
    \begin{align*}
        \miint{Q_\sigma^\lambda(z_0)}\inf_{w\in Q_\sigma^\lambda(z_0)}a(w)^\theta\frac{\Big|u-u_{Q_\sigma^\lambda(z_0)}\Big|^{\theta q}}{\sigma^{\theta q}}\, dz\leq c\miint{Q_\sigma^\lambda(z_0)} H(z,|Du|)^\theta\, dz.
    \end{align*}
\end{lemma}
\begin{proof}
    By Lemmas \ref{lem : semi-Parabolic Poincare inequality} and \ref{lem : last term estimate in Lemma lem : semi-Parabolic Poincare inequality whenever p-intrinsic cylinder}, there exists a positive constant $c=c(\operatorname{data})$ such that 
    \begin{align*}
        &\miint{Q_\sigma^\lambda(z_0)}\inf_{w\in Q_\sigma^\lambda(z_0)}a(w)^\theta\frac{\Big|u-u_{Q_\sigma^\lambda(z_0)}\Big|^{\theta q}}{\sigma^{\theta q}}\, dz \\
        &\quad\leq  c \miint{Q_\sigma^\lambda(z_0)} \inf_{w\in Q_\sigma^\lambda(z_0)}a(w)^\theta|D u|^{\theta q} \,dz\\
        &\qquad +c\inf_{w\in Q_\sigma^\lambda(z_0)}a(w)^\theta\left(\lambda^{2-p}\miint{Q_\sigma^\lambda(z_0)} |Du|^{p-1}\, dz\right)^{\theta q}\\
        &\qquad +c\inf_{w\in Q_\sigma^\lambda(z_0)}a(w)^\theta\left(\lambda^{1-p+\frac{p}{q}}\miint{Q_\sigma^\lambda(z_0)} a(z)^{\frac{q-1}{q}}|Du|^{q-1}\, dz\right)^{\theta q}.
    \end{align*}
    By H\"{o}lder's inequality and \eqref{cond : p-phase condition}, we obtain
    \begin{align*}
        &\miint{Q_\sigma^\lambda(z_0)}\inf_{w\in Q_\sigma^\lambda(z_0)}a(w)^\theta\frac{\Big|u-u_{Q_\sigma^\lambda(z_0)}\Big|^{\theta q}}{\sigma^{\theta q}}\, dz\\ 
        &\quad\leq  c \miint{Q_\sigma^\lambda(z_0)} a(z)^\theta|D u|^{\theta q} \,dz\\
        &\qquad +c\lambda^{(p-q)\theta}\lambda^{(2-p)\theta q}\left(\miint{Q_\sigma^\lambda(z_0)} |Du|^{\theta p}\, dz\right)^{q-\frac{q}{p}}\\
        &\qquad +c\lambda^{(p-q)\theta}\lambda^{\left(1-p+\frac{p}{q}\right)\theta q}\left(\miint{Q_\sigma^\lambda(z_0)} a(z)^\theta|Du|^{\theta q}\, dz\right)^{q-1}\\
        &\quad \leq c \miint{Q_\sigma^\lambda(z_0)} H(z,|Du|)^\theta \,dz.
    \end{align*}
\end{proof}
\subsection{\bf The $(p,q)$-phase case}
By the second and third conditions in \eqref{cond : p,q-phase condition}, we have
$$
\miint{G_{4\rho}^\lambda(z_0)} H_{z_0}(|Du|)\,dz<4a(z_0)\lambda^q
$$
and hence,
\begin{equation}    \label{cond : simple p,q-phase condition}
    \miint{G_{4\rho}^\lambda(z_0)} |Du|^q\,dz<4\lambda^q.
\end{equation}

We now establish a $(p,q)$-intrinsic parabolic Poincar\'{e} inequality.
\begin{lemma}
    Let $u$ be a weak solution to \eqref{eq : main equation} and assume that $G_{4\rho}^\lambda(z_0)\subset \Omega_T$ satisfies \eqref{cond : p,q-phase condition}. Then, for $\sigma\in[2\rho,4\rho]$ and $\theta \in (\frac{q-1}{p},1]$, there exists a positive constant $c=c(n,p,q,L)$ such that
    $$
    \miint{G_\sigma^\lambda (z_0)} H^\theta_{z_0}\left(\frac{\Big|u-u_{G_\sigma^\lambda(z_0)}\Big|}{\sigma}\right)\, dz\leq c \miint{G_\sigma^\lambda(z_0)} H_{z_0}^\theta(|Du|)\, dz.
    $$
\end{lemma}
\begin{proof}
    We only sketch the proof. For a detailed proof, see \cite[Lemma 3.4]{2023_Gradient_Higher_Integrability_for_Degenerate_Parabolic_Double-Phase_Systems}.

    By Lemma \ref{lem : semi-Parabolic Poincare inequality}, \eqref{cond : p,q-phase condition} and \eqref{cond : simple p,q-phase condition}, there exists a positive constant $c=c(n,p,q,L)$ such that
    $$
    \begin{aligned}
        &\miint{G_\sigma^\lambda(z_0)} H_{z_0}^\theta\left(\frac{\Big|u-u_{G_\sigma^\lambda(z_0)}\Big|}{\sigma}\right)dz\\
        &\quad \leq c\miint{G_\sigma^\lambda(z_0)} H_{z_0}^\theta(|Du|)\, dz+cH^\theta_{z_0}\left(\frac{\lambda}{H'_{z_0}(\lambda)}\miint{G_\sigma^\lambda(z_0)} H'_{z_0}(|Du|)\,dz\right)\\
        &\quad \leq c\miint{G_\sigma^\lambda(z_0)} H_{z_0}^\theta(|Du|)\, dz+cH^\theta_{z_0}\left(\lambda^{2-p}\left(\miint{G_\sigma^\lambda(z_0)}|Du|^{q-1}\, dz\right)^{\frac{p-1}{q-1}}\right).\\
    \end{aligned}
    $$
    Here, we obtain from H\"{o}lder inequality and \eqref{cond : simple p,q-phase condition} that
    $$
    \begin{aligned}
        &H^\theta_{z_0}\left(\lambda^{2-p}\left(\miint{G_\sigma^\lambda(z_0)}|Du|^{q-1}\, dz\right)^{\frac{p-1}{q-1}}\right)\\
        &\quad\leq c\left(\lambda^{(2-p)p}\left(\miint{G_\sigma^\lambda(z_0)}|Du|^{q-1}\, dz\right)^{\frac{p(p-1)}{q-1}}\right)^\theta\\
        &\qquad + c\left(a(z_0)\lambda^{(2-p)q}\left(\miint{G_\sigma^\lambda(z_0)}|Du|^{q-1}\, dz\right)^{\frac{q(p-1)}{q-1}}\right)^\theta\\
        &\quad \leq c\miint{G_\sigma^\lambda(z_0)}H^\theta_{z_0}(|Du|)\, dz,
    \end{aligned}
    $$
    and the lemma follows.
\end{proof}
According to \cite{2023_Gradient_Higher_Integrability_for_Degenerate_Parabolic_Double-Phase_Systems}, if we replace $H_{z_0}^\theta(\kappa)$ with $\kappa^{\theta p}$, we get the following result.
\begin{lemma}
    Let $u$ be a weak solution to \eqref{eq : main equation} and assume that $G_{4\rho}^\lambda(z_0)\subset \Omega_T$ satisfies \eqref{cond : p,q-phase condition}. Then, for $\sigma\in[2\rho,4\rho]$ and $\theta \in (\frac{q-1}{p},1]$, there exists a positive constant $c=c(n,p,q,L)$ such that
    $$
    \miint{G_\sigma^\lambda(z_0)}\left(\frac{\Big|u-u_{G_\sigma^\lambda(z_0)}\Big|}{\sigma}\right)^{\theta p}\, dz\leq c\miint{G_\sigma^\lambda(z_0)} |Du|^{\theta p}\, dz.
    $$
\end{lemma}
\section{\bf Reverse H\"{o}lder inequalities}\label{section 5}
In this section, we establish reverse H\"{o}lder inequalities separately in each intrinsic cylinder. For this, we need the following auxiliary lemmas, called the Gagliardo-Nirenberg inequality and a standard iteration lemma, see \cite[Lemma 2.12]{Hasto_2021} and \cite[Lemma 6.1]{2003_Giusti_Direct_methods_in_the_calculus_of_variations}, respectively. 
\begin{lemma}   \label{lem : Gagliardo-Nirenberg inequality}
    For an open ball $B_{\rho}(x_0)\subset \mr^n$, take $p_1,\,p_2,\,p_3\in[1,\infty)$, $\vartheta\in(0,1)$ and let $\psi\in W^{1,p_2}(B)$. Suppose that
    $$
    -\frac{n}{p_1}\leq \vartheta\left(1-\frac{n}{p_2}\right)-(1-\vartheta)\frac{n}{p_3}.
    $$
    Then there exists a positive constant $c=c(n,p_1)$ such that
    $$
    \dashint_B \frac{|\psi|^{p_1}}{\rho^{p_1}}\, dx\leq c \left(\dashint_{B_\rho(x_0)}\left[\frac{|\psi|^{p_2}}{\rho^{p_2}}+|D\psi|^{p_2}\right]dx\right)^{\frac{\vartheta p_1}{p_2}}\left(\dashint_{B_{\rho}(x_0)}\frac{|\psi|^{p_3}}{\rho^{p_3}}\,dx\right)^{\frac{(1-\vartheta)p_1}{p_3}}.
    $$
\end{lemma}
\begin{lemma}   \label{lem : a standard iteration lemma}
    Let $0<\rho<\tau<\infty$, and let $g:[\rho,\tau]\rightarrow [0,\infty)$ be a bounded function. Suppose that
    $$
    g(\rho_1)\leq \vartheta g(\rho_2)+\frac{A}{(\rho_2-\rho_1)^\gamma}+B
    $$
    holds for all $0<\rho\leq \rho_1<\rho_2\leq \tau$, where $\vartheta \in (0,1)$, $A,B\geq 0$ and $\gamma>0$. Then there exists a positive constant $c$ depending on $\vartheta$ and $\gamma$ such that
    $$
    g(\rho)\leq c\left(\frac{A}{(\tau-\rho)^{\gamma}}+B\right).
    $$
\end{lemma}

\subsection{The $p$-phase case.}
In this case, we consider the $p$-intrinsic cylinder denoted in \eqref{eq : definition of p-intrinsic cylinder}. We write
$$
S(u,Q_{\rho}^\lambda(z_0))=\sup_{I_\rho^\lambda(t_0)}\dashint_{B_\rho(x_0)}\frac{\Big|u-u_{Q_\rho^\lambda(z_0)}\Big|^2}{\rho^2}\, dx.
$$
\begin{lemma}   \label{lem : estimation of S in p-intrinsic cylinder}
    Let $u$ be a weak solution to \eqref{eq : main equation} with \eqref{cond : main assumption with infty} or \eqref{cond : main assumption with s}. Then there exists a constant $c=c(\operatorname{data})>1$ such that
    $$
    S(u,Q_{2\rho}^\lambda(z_0))\leq c \lambda^2,
    $$
    whenever $Q_{2\kappa\rho}^\lambda(z_0)\subset \Omega_T$ satisfies \eqref{cond : p-phase condition}.
\end{lemma}
\begin{proof}
    Let $2\rho\leq\rho_1<\rho_2\leq 4\rho$. By Lemma \ref{lem : Caccioppoli inequality}, there exists a positive constant $c=c(n,p,q,\nu,L)$ such that
    \begin{align}
        &\lambda^{p-2}S(u,Q_{\rho_1}^\lambda(z_0))\nonumber\\
        &\quad \;\leq \frac{c\rho_2^q}{(\rho_2-\rho_1)^q}\miint{Q_{\rho_2}^\lambda (z_0)}\left[\frac{\Big|u-u_{Q_{\rho_2}^\lambda(z_0)}\Big|^p}{\rho_2^p}+a(z)\frac{\Big|u-u_{Q_{\rho_2}^\lambda(z_0)}\Big|^q}{\rho_2^q}\right] dz\nonumber\\\label{eq : estimation of S in p-intrinsic cylinder}
        &\qquad +\frac{c\rho_2^2\lambda^{p-2}}{(\rho_2-\rho_1)^2}\miint{Q_{\rho_2}^\lambda(z_0)}\frac{\Big|u-u_{Q_{\rho_2}^\lambda(z_0)}\Big|^2}{\rho_2^2}\, dz.
    \end{align}
    To estimate the $p$-term in the first term on the right-hand side of \eqref{eq : estimation of S in p-intrinsic cylinder}, we use Lemma \ref{lem : p-intrinsic parabolic Poincare inequality of p-term in p-intrinsic cylinder} with $\theta=1$ and \eqref{cond : p-phase condition}$_2$. Then we have
    \begin{equation}    \label{eq : estimation u^p in p-intrinsic cylinder}
    \miint{Q_{\rho_2}^\lambda(z_0)}\frac{\Big|u-u_{Q_{\rho_2}^\lambda(z_0)}\Big|^p}{\rho_2^p}\,dz\leq c\lambda^p
    \end{equation}
    for some $c=c(\operatorname{data})>1$. Now, we estimate the $q$-term in the first term. Since $a(\cdot) \in C^{\alpha,\frac{\alpha}{2}}(\Omega_T)$, we get
    $$
    \begin{aligned}
        \miint{Q_{\rho_2}^\lambda(z_0)} a(z)\frac{\Big|u-u_{Q_{\rho_2}^\lambda(z_0)}\Big|^q}{\rho_2^q}\, dz &\leq \miint{Q_{\rho_2}^\lambda(z_0)} \inf_{w\in Q_{\rho_2}^\lambda (z_0)} a(w)\frac{\Big|u-u_{Q_{\rho_2}^\lambda (z_0)}\Big|^q}{\rho_2^q}\, dz\\
        &\quad +[a]_\alpha \rho_2^\alpha \miint{Q_{\rho_2}^\lambda(z_0)} \frac{\Big|u-u_{Q_{\rho_2}^\lambda (z_0)}\Big|^q}{\rho_2^q}\, dz. 
    \end{aligned}
    $$
    By Lemma \ref{lem : p-intrinsic parabolic Poincare inequality of q-term in p-intrinsic cylinder} with $\theta = 1$ and \eqref{cond : p-phase condition}$_2$, we obtain
    $$
    \miint{Q_{\rho_2}^\lambda(z_0)} \inf_{w\in Q_{\rho_2}^\lambda (z_0)} a(w)\frac{\Big|u-u_{Q_{\rho_2}^\lambda (z_0)}\Big|^q}{\rho_2^q}\, dz\leq c \lambda^p
    $$
    for some $c=c(\operatorname{data})>1$. The remaining part is proven by dividing into cases, where either \eqref{cond : main assumption with infty} or \eqref{cond : main assumption with s} is satisfied. When \eqref{cond : main assumption with infty} holds, by \eqref{eq : estimation of u-mean of u as u} and \eqref{eq : estimation u^p in p-intrinsic cylinder}, we have
    $$
    \begin{aligned}
        \rho_2^\alpha \miint{Q_{\rho_2}^\lambda(z_0)}\frac{\Big|u-u_{Q_{\rho_2}^\lambda (z_0)}\Big|^q}{\rho_2^q}\, dz &\leq c\rho_2^\alpha \miint{Q_{\rho_2}^\lambda(z_0)}\frac{\Big|u-u_{Q_{\rho_2}^\lambda (z_0)}\Big|^{p}}{\rho_2^p}\cdot \frac{|u|^{q-p}}{\rho_2^{q-p}}\, dz\\
        &\leq c\rho_2^{\alpha+p-q}\|u\|_{L^\infty(\Omega_T)}^{q-p}\miint{Q_{\rho_2}^\lambda(z_0)}\frac{\Big|u-u_{Q_{\rho_2}^\lambda (z_0)}\Big|^{p}}{\rho_2^p}\, dz\\
        &\leq c\lambda^p
    \end{aligned}
    $$
    for some $c=c(\operatorname{data})$. Now, we assume that \eqref{cond : main assumption with s} holds. To estimate, we use Lemma \ref{lem : Gagliardo-Nirenberg inequality} with $p_1=q$, $p_2=p$, $p_3=s$ and $\vartheta=\frac{p}{q}$. For this, we check
    $$
    -\frac{n}{q}\leq \frac{p}{q}\left(1-\frac{n}{p}\right)-\left(1-\frac{p}{q}\right)\frac{n}{s}.
    $$
    Indeed, since $q\leq p +\frac{s\alpha}{n+s} \leq p+\frac{s}{n+s}$ and $2\leq q$, we get $0\leq qs-s\leq ps-qn+pn$, and hence 
    $$
    -\frac{n}{q}\leq -\frac{n}{q}+\frac{ps-qn+pn}{qs}=\frac{p}{q}\left(1-\frac{n}{p}\right)-\left(1-\frac{p}{q}\right)\frac{n}{s}.
    $$
    Thus, by Lemma \ref{lem : Gagliardo-Nirenberg inequality} with $p_1=q$, $p_2=p$, $p_3=s$ and $\vartheta=\frac{p}{q}$, we obtain
    $$
    \begin{aligned}
        &\rho_2^\alpha \miint{Q_{\rho_2}^\lambda(z_0)}\frac{\Big|u-u_{Q_{\rho_2}^\lambda (z_0)}\Big|^q}{\rho_2^q}\, dz\\
        &\qquad \leq c\rho_2^\alpha\left(\miint{Q_{\rho_2}^\lambda(z_0)}\left[\frac{\Big|u-u_{Q_{\rho_2}^\lambda (z_0)}\Big|^p}{\rho_2^p}+|Du|^p\right] dz\right)\\
        &\qquad \qquad \times \left(\sup_{I_{\rho_2}^\lambda(t_0)}\dashint_{B_{\rho_2}(x_0)}\frac{\Big|u-u_{Q_{\rho_2}^\lambda(z_0)}\Big|^s}{\rho_2^s}\,dx\right)^{\frac{q-p}{s}}.
    \end{aligned}
    $$
    By \eqref{eq : estimation of u-mean of u as u} and \eqref{cond : main assumption with s}, we get
    $$
    \begin{aligned}
        &\rho_2^\alpha \left(\sup_{I_{\rho_2}^\lambda(t_0)}\dashint_{B_{\rho_2}(x_0)}\frac{\Big|u-u_{Q_{\rho_2}^\lambda(z_0)}\Big|^s}{\rho_2^s}\,dx\right)^{\frac{q-p}{s}}\leq c\rho_2^\alpha \left(\sup_{I_{\rho_2}^\lambda(t_0)}\dashint_{B_{\rho_2}(x_0)}\frac{|u|^s}{\rho_2^s}\,dx\right)^{\frac{q-p}{s}}\\
        &\qquad \leq c\rho_2^{\alpha-\frac{(q-p)(n+s)}{s}} \left(\sup_{I_{\rho_2}^\lambda(t_0)}\int_{\Omega} |u|^s\,dx\right)^{\frac{q-p}{s}}\leq c
    \end{aligned}
    $$
    for some $c=c(n,p,q,s,\alpha,\operatorname{diam}(\Omega),\|u\|_{L^\infty(0,T;L^s(\Omega))})>1$. Therefore, by \eqref{eq : estimation u^p in p-intrinsic cylinder}, we have
    \begin{equation}\label{eq : estimation u^q in p-intrinsic cylinder}
    \rho_2^\alpha \miint{Q_{\rho_2}^\lambda(z_0)}\frac{\Big|u-u_{Q_{\rho_2}^\lambda (z_0)}\Big|^q}{\rho_2^q}\, dz\leq c \lambda^p
    \end{equation}
    for some $c=c(\operatorname{data})>1$. Lastly, applying the method in \cite[Lemma 4.3]{2023_Gradient_Higher_Integrability_for_Degenerate_Parabolic_Double-Phase_Systems} to the second term on the right-hand side of \eqref{eq : estimation of S in p-intrinsic cylinder} yields
    \begin{equation}    \label{eq : estimation u^2 in p-intrinsic cylinder}
    \miint{Q_{\rho_2}^\lambda(z_0)}\frac{\Big|u-u_{Q_{\rho_2}^\lambda(z_0)}\Big|^2}{\rho_2^2}\, dz\leq c \lambda S(u,Q_{\rho_2}^\lambda(z_0))^{\frac{1}{2}}
    \end{equation}
    for some $c=c(\operatorname{data})>1$. Combining \eqref{eq : estimation of S in p-intrinsic cylinder}, \eqref{eq : estimation u^p in p-intrinsic cylinder}, \eqref{eq : estimation u^q in p-intrinsic cylinder} and \eqref{eq : estimation u^2 in p-intrinsic cylinder} gives
    $$
    S(u,Q_{\rho_1}^\lambda(z_0))\leq \frac{c\rho_2^q}{(\rho_2-\rho_1)^q}\lambda^2+\frac{c\rho_2^2}{(\rho_2-\rho_1)^2}\lambda S(u,Q_{\rho_2}^\lambda(z_0))^{\frac{1}{2}}.
    $$
    Then it follows from Young's inequality that
    $$
    S(u,Q_{\rho_1}^\lambda(z_0))\leq \frac{1}{2}S(u,Q_{\rho_2}^\lambda(z_0))+c\left(\frac{\rho_2^q}{(\rho_2-\rho_1)^q}+\frac{\rho_2^4}{(\rho_2-\rho_1)^4}\right)\lambda^2.
    $$
    Finally, Lemma \ref{lem : a standard iteration lemma} yields the conclusion.
\end{proof}
\begin{lemma}\label{lem : estimation of difference of u and mean u in p intrinsic cylinder with s<infty}
    Let $u$ be a weak solution to \eqref{eq : main equation} with \eqref{cond : main assumption with s}. Then there exist constants $c=c(\operatorname{data}_s)>1$ and $\theta_1=\theta_1(n,p,q,s)\in (0,1)$ such that for any $\theta\in(\theta_1,1)$
    $$
    \begin{aligned}
        &\miint{Q_{2\rho}^\lambda (z_0)}\left[\frac{\Big|u-u_{Q_{2\rho}^\lambda (z_0)}\Big|^p}{(2\rho)^p}+a(z)\frac{\Big|u-u_{Q_{2\rho}^\lambda (z_0)}\Big|^q}{(2\rho)^q}\right] dz\\
        &\quad \leq c\lambda^{(1-\theta)p}\miint{Q_{2\rho}^\lambda (z_0)} \left[\frac{\Big|u-u_{Q_{2\rho}^\lambda (z_0)}\Big|^{\theta p}}{(2\rho)^{\theta p}}+|Du|^{\theta p}\right] dz\\
        &\qquad + c\lambda^{(1-\theta)p}\miint{Q_{2\rho}^\lambda (z_0)}\inf_{w\in Q_{2\rho}^\lambda (z_0)}a(w)^\theta \left[\frac{\Big|u-u_{Q_{2\rho}^\lambda (z_0)}\Big|^{\theta q}}{(2\rho)^{\theta q}}+|Du|^{\theta q}\right] dz,
    \end{aligned}
    $$
    whenever $Q_{2\kappa \rho}^\lambda (z_0)\subset \Omega_T$ satisfies \eqref{cond : p-phase condition}.
\end{lemma}
\begin{proof}
    Since $a(\cdot) \in C^{\alpha,\frac{\alpha}{2}}(\Omega_T)$, we see that
    \begin{align}
        &\miint{Q_{2\rho}^\lambda (z_0)}\left[\frac{\Big|u-u_{Q_{2\rho}^\lambda (z_0)}\Big|^p}{(2\rho)^p}+a(z)\frac{\Big|u-u_{Q_{2\rho}^\lambda (z_0)}\Big|^q}{(2\rho)^q}\right] dz\nonumber\\
        &\quad \leq \miint{Q_{2\rho}^\lambda (z_0)}\frac{\Big|u-u_{Q_{2\rho}^\lambda (z_0)}\Big|^p}{(2\rho)^p}\, dz+ \miint{Q_{2\rho}^\lambda (z_0)} \inf_{w\in Q_{2\rho}^\lambda (z_0)} a(w)\frac{\Big|u-u_{Q_{2\rho}^\lambda (z_0)}\Big|^q}{(2\rho)^q}\, dz\nonumber\\ \label{eq : energy estimate in p-intrinsic cylinder}
        &\qquad +[a]_\alpha (2\rho)^\alpha \miint{Q_{2\rho}^\lambda (z_0)} \frac{\Big|u-u_{Q_{2\rho}^\lambda (z_0)}\Big|^q}{(2\rho)^q}\, dz.
    \end{align}
    By applying Lemma \ref{lem : Gagliardo-Nirenberg inequality} in the same manner as in \cite[Lemma 4.4]{2023_Gradient_Higher_Integrability_for_Degenerate_Parabolic_Double-Phase_Systems}, we obtain that for any $\theta\in\left(\frac{n}{n+2},1\right)$,
    \begin{align*}
        &\miint{Q_{2\rho}^\lambda (z_0)}\frac{\Big|u-u_{Q_{2\rho}^\lambda (z_0)}\Big|^p}{(2\rho)^p}\, dz+ \miint{Q_{2\rho}^\lambda (z_0)} \inf_{w\in Q_{2\rho}^\lambda (z_0)} a(w)\frac{\Big|u-u_{Q_{2\rho}^\lambda (z_0)}\Big|^q}{(2\rho)^q}\, dz\\
        &\quad \leq c\miint{Q_{2\rho}^\lambda (z_0)}\left[\frac{\Big|u-u_{Q_{2\rho}^\lambda (z_0)}\Big|^{\theta p}}{(2\rho)^{\theta p}}+|Du|^{\theta p}\right] dz\left(S(u,Q_{2\rho}^\lambda(z_0))\right)^{\frac{(1-\theta)p}{2}}\\
        &\qquad + c\miint{Q_{2\rho}^\lambda (z_0)}\left[\inf_{w\in Q_{2\rho}^\lambda (z_0)}a(w)^\theta \frac{\Big|u-u_{Q_{2\rho}^\lambda (z_0)}\Big|^{\theta q}}{(2\rho)^{\theta q}}+\inf_{w\in Q_{2\rho}^\lambda (z_0)}a(w)^\theta |Du|^{\theta q}\right] dz\\
        &\qquad\quad \times \lambda^{(p-q)(1-\theta)}\left(S(u,Q_{2\rho}^\lambda(z_0))\right)^{\frac{(1-\theta)q}{2}}
    \end{align*}
    for some $c=c(\operatorname{data}_s)>1$. It follows from Lemma \ref{lem : estimation of S in p-intrinsic cylinder} that
    \begin{equation*}
        \begin{aligned}
            &\miint{Q_{2\rho}^\lambda (z_0)}\frac{\Big|u-u_{Q_{2\rho}^\lambda (z_0)}\Big|^p}{(2\rho)^p}\, dz+ \miint{Q_{2\rho}^\lambda (z_0)} \inf_{w\in Q_{2\rho}^\lambda (z_0)} a(w)\frac{\Big|u-u_{Q_{2\rho}^\lambda (z_0)}\Big|^q}{(2\rho)^q}\, dz\\
            &\quad \leq c\lambda^{(1-\theta)p}\miint{Q_{2\rho}^\lambda (z_0)}\left[\frac{\Big|u-u_{Q_{2\rho}^\lambda (z_0)}\Big|^{\theta p}}{(2\rho)^{\theta p}}+|Du|^{\theta p}\right] dz\\
            &\qquad + c\lambda^{(1-\theta)p}\miint{Q_{2\rho}^\lambda (z_0)}\inf_{w\in Q_{2\rho}^\lambda (z_0)}a(w)^\theta \left[\frac{\Big|u-u_{Q_{2\rho}^\lambda (z_0)}\Big|^{\theta q}}{(2\rho)^{\theta q}}+|Du|^{\theta q}\right] dz.
        \end{aligned}
    \end{equation*}

    Now, it remains to estimate only the last term in \eqref{eq : energy estimate in p-intrinsic cylinder}. We claim
    \begin{equation}\label{eq : q-energy estimate in p-intrinsic cylinder}
    \begin{aligned}
    &(2\rho)^\alpha\miint{Q^\lambda_{2\rho}(z_0)} \frac{|u-u_{Q^\lambda_{2\rho}(z_0)}|^q}{(2\rho)^q}\, dz\\
    &\qquad \leq c \lambda ^{p(1-\theta)}\miint{Q^\lambda_{2\rho}(z_0)} \left[\frac{\Big|u-u_{Q^\lambda_{2\rho}(z_0)}\Big|^{\theta p}}{(2\rho)^{\theta p}}+|Du|^{\theta p}\right] dz.
    \end{aligned}
    \end{equation} 
    We shall prove the above claim by dividing it into cases $2\leq s \leq 4$ and $4<s<\infty$. Suppose that $s\in [2,4]$. 
    We then note that 
    $$
    \frac{nq}{(n+2)p}\leq \frac{n}{n+2}\left(1+\frac{s}{(n+s)p}\right)=\frac{n}{n+2}\cdot\frac{pn+(p+1)s}{(n+s)p}<1.
    $$
    Since 
    \begin{equation}\label{eq : assumption of Gagliardo Nirenberg Lemma}
        -\frac{n}{q}\leq \frac{\theta p}{q}\left(1-\frac{n}{\theta p}\right)-\left(1-\frac{\theta p}{q}\right)\frac{n}{p_3} \ \, \iff \ \, \frac{nq}{(n+p_3)p}\leq \theta,
    \end{equation}
    by taking $p_1=q$, $p_2=\theta p$, $p_3=2$ and $\vartheta=\frac{\theta p}{q}$ for $\theta\in\left(\frac{nq}{(n+2)p},1\right)$ in Lemma \ref{lem : Gagliardo-Nirenberg inequality}, we find that the assumption of Lemma \ref{lem : Gagliardo-Nirenberg inequality} is satisfied. Hence, we have
    $$
    \begin{aligned}
        &(2\rho)^\alpha\miint{Q^\lambda_{2\rho}(z_0)} \frac{\Big|u-u_{Q^\lambda_{2\rho}(z_0)}\Big|^q}{(2\rho)^q}\, dz\\
        &\qquad \leq c(2\rho)^\alpha\miint{Q_{2\rho}^\lambda(z_0)}\left[\frac{\Big|u-u_{Q^\lambda_{2\rho}(z_0)}\Big|^{\theta p}}{(2\rho)^{\theta p}}+|Du|^{\theta p}\right] dz\\
        &\qquad \qquad \times \sup_{I_{2\rho}^\lambda}\left(\dashint_{B_{2\rho}(x_0)}\frac{\Big|u-u_{Q^\lambda_{2\rho}(z_0)}\Big|^{2}}{(2\rho)^2}\,dx\right)^{\frac{q-\theta p}{2}}\\
        &\qquad \leq c\miint{Q_{2\rho}^\lambda(z_0)}\left[\frac{\Big|u-u_{Q^\lambda_{2\rho}(z_0)}\Big|^{\theta p}}{(2\rho)^{\theta p}}+|Du|^{\theta p}\right] dz\\
        &\qquad \qquad \times (2\rho)^\alpha \sup_{I_{2\rho}^\lambda}\left(\dashint_{B_{2\rho}(x_0)}\frac{\Big|u-u_{Q^\lambda_{2\rho}(z_0)}\Big|^{s}}{(2\rho)^s}\,dx\right)^{\frac{q-p}{s}}\\
        &\qquad \qquad \times \sup_{I_{2\rho}^\lambda}\left(\dashint_{B_{2\rho}(x_0)}\frac{\Big|u-u_{Q^\lambda_{2\rho}(z_0)}\Big|^{2}}{(2\rho)^2}\,dx\right)^{\frac{p(1-\theta)}{2}}\\
        &\qquad \leq c\lambda^{(1-\theta)p}\miint{Q_{2\rho}^\lambda (z_0)}\left[\frac{\Big|u-u_{Q_{2\rho}^\lambda (z_0)}\Big|^{\theta p}}{(2\rho)^{\theta p}}+|Du|^{\theta p}\right] dz
    \end{aligned}
    $$
    for some $c=c(\operatorname{data}_s)>1$. Finally, we assume that $4<s<\infty$. Consider $$\theta \in \left(\frac{ps(s-3)-2(q-p)}{ps(s-3)},1\right)\quad \text{and} \quad \tilde{p}=\frac{2s(q-p\theta)}{ps(1-\theta)+2(q-p)}.$$
    Since $q-p< 1$ and $s>4$, $2(q-p)<2<ps(s-3)$ and $\tilde{p}<s$. Moreover, through a simple calculation, we easily obtain that $s-1<\tilde{p}$ whenever $\frac{ps(s-3)-2(q-p)}{ps(s-3)}<\theta$. Hence, if
    $$\theta \in \left(\frac{ps(s-3)-2(q-p)}{ps(s-3)},1\right)\quad \text{and} \quad \tilde{p}=\frac{2s(q-p\theta)}{ps(1-\theta)+2(q-p)},$$
    then $s-1<\tilde{p}<s$. Since $2<s-1<\tilde{p}$ and $2\leq p$, we have 
    $$
    \begin{aligned}
        \frac{nq}{(n+\tilde{p})p}<\frac{nq}{(n+s-1)p}&<\frac{n}{n+s-1}\left(1+\frac{s}{(n+s)p}\right)\\
        &\leq \frac{n}{n+s-1}\left(\frac{2n+3s}{2n+2s}\right)<1.
    \end{aligned}
    $$
    Hence, by letting $p_1=q$, $p_2=\theta p$, $p_3=\tilde{p}$ and $\vartheta=\frac{\theta p}{q}$, we observe from \eqref{eq : assumption of Gagliardo Nirenberg Lemma} that the assumption in Lemma \ref{lem : Gagliardo-Nirenberg inequality} is satisfied for any $\theta \in \left(\frac{ps(s-3)-2(q-p)}{ps(s-3)},1\right)$. Thus, we have
    $$
    \begin{aligned}
        &(2\rho)^\alpha\miint{Q^\lambda_{2\rho}(z_0)} \frac{\Big|u-u_{Q^\lambda_{2\rho}(z_0)}\Big|^q}{(2\rho)^q}\, dz\\
        &\qquad\leq c\miint{Q^\lambda_{2\rho}(z_0)} \left[\frac{\Big|u-u_{Q^\lambda_{2\rho}(z_0)}\Big|^{\theta p}}{(2\rho)^{\theta p}}+|Du|^{\theta p}\right] dz\\
        &\qquad\qquad\times(2\rho)^\alpha \left(\sup_{I_{2\rho}^\lambda (t_0)}\dashint_{B_{2\rho}(x_0)}\frac{\Big|u-u_{Q^\lambda_{2\rho}(z_0)}\Big|^{\tilde{p}}}{(2\rho)^{\tilde{p}}}\,dx\right)^{\frac{q-p\theta}{\tilde{p}}}.
    \end{aligned} 
    $$
    By the interpolation inequality for $L^p$-norms, we get
    $$
    \begin{aligned}
    &(2\rho)^\alpha \sup_{I_{2\rho}^\lambda (t_0)}\left(\dashint_{B_{2\rho}(x_0)}\frac{\Big|u-u_{Q^\lambda_{2\rho}(z_0)}\Big|^{\tilde{p}}}{(2\rho)^{\tilde{p}}}\,dx\right)^{\frac{q-p\theta}{\tilde{p}}}\\
    &\qquad\leq \sup_{I_{2\rho}^\lambda (t_0)}\left(\dashint_{B_{2\rho}(x_0)}\frac{\Big|u-u_{Q^\lambda_{2\rho}(z_0)}\Big|^{2}}{(2\rho)^{2}}\,dx\right)^{\frac{q-p\theta}{2}\tilde{\theta}}\\
    &\qquad \qquad \times (2\rho)^\alpha \sup_{I_{2\rho}^\lambda (t_0)}\left(\dashint_{B_{2\rho}(x_0)}\frac{\Big|u-u_{Q^\lambda_{2\rho}(z_0)}\Big|^{s}}{(2\rho)^{s}}\,dx\right)^{\frac{q-p\theta}{s}(1-\tilde{\theta})},
    \end{aligned}
    $$
    where $\tilde{\theta}=\frac{2(s-\tilde{p})}{\tilde{p}(s-2)}$ and $1-\tilde{\theta}=\frac{s(\tilde{p}-2)}{\tilde{p}(s-2)}$. We obtain from Lemma \ref{lem : estimation of S in p-intrinsic cylinder}, \eqref{cond : main assumption with s} and definition of $\|u\|_{C(0,T;L^s(\Omega))}$ that
    $$
    (2\rho)^\alpha \sup_{I_{2\rho}^\lambda (t_0)}\left(\dashint_{B_{2\rho}(x_0)}\frac{\Big|u-u_{Q^\lambda_{2\rho}(z_0)}\Big|^{\tilde{p}}}{(2\rho)^{\tilde{p}}}\,dx\right)^{\frac{q-p\theta}{\tilde{p}}}\leq c\lambda^{\frac{2(q-p\theta)(s-\tilde{p})}{\tilde{p}(s-2)}}\rho^{\alpha-\frac{s(n+s)(q-p\theta)(\tilde{p}-2)}{\tilde{p}s(s-2)}}.
    $$
    Since $\frac{2(q-p\theta)(s-\tilde{p})}{\tilde{p}(s-2)}=p(1-\theta)$ and $\alpha-\frac{s(n+s)(q-p\theta)(\tilde{p}-2)}{\tilde{p}s(s-2)}=\alpha-\frac{(n+s)(q-p)}{s}\geq 0$, we have
    $$
    (2\rho)^\alpha \sup_{I_{2\rho}^\lambda (t_0)}\left(\dashint_{B_{2\rho}(x_0)}\frac{\Big|u-u_{Q^\lambda_{2\rho}(z_0)}\Big|^{\tilde{p}}}{(2\rho)^{\tilde{p}}}\,dx\right)^{\frac{q-p\theta}{\tilde{p}}}\leq c\lambda^{p(1-\theta)}.
    $$
    Therefore, the claim holds in both cases and hence we have the conclusion for 
    $$
    \theta_1=\begin{cases}
        \frac{nq}{p(n+2)}\quad&\text{if }2\leq s\leq 4,\\
        \\
        \max\left\{\frac{n}{n+2},\frac{ps(s-3)-2(q-p)}{ps(s-3)}\right\}\quad &\text{if }4<s<\infty.
    \end{cases}
    $$
\end{proof}
\begin{lemma}\label{lem : estimation of difference of u and mean u in p intrinsic cylinder with s=infty}
    Let $u$ be a weak solution to \eqref{eq : main equation} with \eqref{cond : main assumption with infty}. Then there exist constants $c=c(\operatorname{data}_b)>1$ and $\theta_1=\theta_1(n)\in (0,1)$ such that for any $\theta\in(\theta_1,1)$
    $$
    \begin{aligned}
        &\miint{Q_{2\rho}^\lambda (z_0)}\left[\frac{\Big|u-u_{Q_{2\rho}^\lambda (z_0)}\Big|^p}{(2\rho)^p}+a(z)\frac{\Big|u-u_{Q_{2\rho}^\lambda (z_0)}\Big|^q}{(2\rho)^q}\right] dz\\
        &\quad \leq c\lambda^{(1-\theta)p}\miint{Q_{2\rho}^\lambda (z_0)} \left[\frac{\Big|u-u_{Q_{2\rho}^\lambda (z_0)}\Big|^{\theta p}}{(2\rho)^{\theta p}}+|Du|^{\theta p}\right] dz\\
        &\qquad + c\lambda^{(1-\theta)p}\miint{Q_{2\rho}^\lambda (z_0)}\inf_{w\in Q_{2\rho}^\lambda (z_0)}a(w)^\theta \left[\frac{\Big|u-u_{Q_{2\rho}^\lambda (z_0)}\Big|^{\theta q}}{(2\rho)^{\theta q}}+|Du|^{\theta q}\right] dz,
    \end{aligned}
    $$
    whenever $Q_{2\kappa \rho}^\lambda (z_0)\subset \Omega_T$ satisfies \eqref{cond : p-phase condition}.
\end{lemma}
\begin{proof}
    As in the above proof, we have 
    \begin{align*}
        &\miint{Q_{2\rho}^\lambda (z_0)}\left[\frac{\Big|u-u_{Q_{2\rho}^\lambda (z_0)}\Big|^p}{(2\rho)^p}+a(z)\frac{\Big|u-u_{Q_{2\rho}^\lambda (z_0)}\Big|^q}{(2\rho)^q}\right] dz\\
        &\quad \leq \miint{Q_{2\rho}^\lambda (z_0)}\frac{\Big|u-u_{Q_{2\rho}^\lambda (z_0)}\Big|^p}{(2\rho)^p}\, dz+ \miint{Q_{2\rho}^\lambda (z_0)} \inf_{w\in Q_{2\rho}^\lambda (z_0)} a(w)\frac{\Big|u-u_{Q_{2\rho}^\lambda (z_0)}\Big|^q}{(2\rho)^q}\, dz\\ 
        &\qquad +[a]_\alpha (2\rho)^\alpha \miint{Q_{2\rho}^\lambda (z_0)} \frac{\Big|u-u_{Q_{2\rho}^\lambda (z_0)}\Big|^q}{(2\rho)^q}\, dz.
    \end{align*}
and, for any $\theta\in\left(\frac{n}{n+2},1\right)$,
\begin{equation}\label{eq : p,q-energy estimate in p-intrinsic cylinder}
        \begin{aligned}
            &\miint{Q_{2\rho}^\lambda (z_0)}\frac{\Big|u-u_{Q_{2\rho}^\lambda (z_0)}\Big|^p}{(2\rho)^p}\, dz+ \miint{Q_{2\rho}^\lambda (z_0)} \inf_{w\in Q_{2\rho}^\lambda (z_0)} a(w)\frac{\Big|u-u_{Q_{2\rho}^\lambda (z_0)}\Big|^q}{(2\rho)^q}\, dz\\
            &\quad \leq c\lambda^{(1-\theta)p}\miint{Q_{2\rho}^\lambda (z_0)}\left[\frac{\Big|u-u_{Q_{2\rho}^\lambda (z_0)}\Big|^{\theta p}}{(2\rho)^{\theta p}}+|Du|^{\theta p}\right] dz\\
            &\qquad + c\lambda^{(1-\theta)p}\miint{Q_{2\rho}^\lambda (z_0)}\inf_{w\in Q_{2\rho}^\lambda (z_0)}a(w)^\theta \left[\frac{\Big|u-u_{Q_{2\rho}^\lambda (z_0)}\Big|^{\theta q}}{(2\rho)^{\theta q}}+|Du|^{\theta q}\right] dz.
        \end{aligned}
\end{equation}
Moreover, we obtain from $q\leq p + \alpha$ and \eqref{eq : p,q-energy estimate in p-intrinsic cylinder} that
    $$
    \begin{aligned}
    &(2\rho)^\alpha\miint{Q^\lambda_{2\rho}(z_0)} \frac{\Big|u-u_{Q^\lambda_{2\rho}(z_0)}\Big|^q}{(2\rho)^q}\, dz\\
    &\qquad =(2\rho)^\alpha\miint{Q^\lambda_{2\rho}(z_0)} \frac{\Big|u-u_{Q^\lambda_{2\rho}(z_0)}\Big|^{q-p}}{(2\rho)^{q-p}}\cdot\frac{\Big|u-u_{Q^\lambda_{2\rho}(z_0)}\Big|^p}{(2\rho)^p}\, dz\\
    &\qquad \leq c\rho^{\alpha+p-q} \miint{Q^\lambda_{2\rho}(z_0)}\frac{\Big|u-u_{Q^\lambda_{2\rho}(z_0)}\Big|^p}{(2\rho)^p}\, dz\\
    &\qquad\leq c\lambda^{(1-\theta)p}\miint{Q_{2\rho}^\lambda (z_0)}\left[\frac{\Big|u-u_{Q_{2\rho}^\lambda (z_0)}\Big|^{\theta p}}{(2\rho)^{\theta p}}+|Du|^{\theta p}\right] dz\\
    &\qquad\quad + c\lambda^{(1-\theta)p}\miint{Q_{2\rho}^\lambda (z_0)}\inf_{w\in Q_{2\rho}^\lambda (z_0)}a(w)^\theta \left[\frac{\Big|u-u_{Q_{2\rho}^\lambda (z_0)}\Big|^{\theta q}}{(2\rho)^{\theta q}}+|Du|^{\theta q}\right] dz.
    \end{aligned}
    $$
    for some $c=c(p,q,\alpha,\operatorname{diam}(\Omega),\|u\|_{L^\infty(\Omega_T)})>1$. Hence, combining the above inequalities, we have, for any $\theta\in\left(\theta_1,1\right)$,
    \begin{align*}
        &\miint{Q_{2\rho}^\lambda (z_0)}\left[\frac{\Big|u-u_{Q_{2\rho}^\lambda (z_0)}\Big|^p}{(2\rho)^p}+a(z)\frac{\Big|u-u_{Q_{2\rho}^\lambda (z_0)}\Big|^q}{(2\rho)^q}\right] dz\\
        &\qquad\leq c\lambda^{(1-\theta)p}\miint{Q_{2\rho}^\lambda (z_0)}\left[\frac{\Big|u-u_{Q_{2\rho}^\lambda (z_0)}\Big|^{\theta p}}{(2\rho)^{\theta p}}+|Du|^{\theta p}\right] dz\\
        &\qquad\quad + c\lambda^{(1-\theta)p}\miint{Q_{2\rho}^\lambda (z_0)}\inf_{w\in Q_{2\rho}^\lambda (z_0)}a(w)^\theta \left[\frac{\Big|u-u_{Q_{2\rho}^\lambda (z_0)}\Big|^{\theta q}}{(2\rho)^{\theta q}}+|Du|^{\theta q}\right] dz,
    \end{align*}
    where $\theta_1 := \frac{n}{n+2}\in(0,1)$ and $c=c(\operatorname{data}_b)>1$.
\end{proof}
\begin{lemma}
    Let $u$ be a weak solution to \eqref{eq : main equation} with \eqref{cond : main assumption with infty} or \eqref{cond : main assumption with s}. Then there exist constants $c=c(\operatorname{data})>1$ and $\theta_0\in (0,1)$ such that, for any $\theta\in(\theta_0,1)$,
    $$
    \begin{aligned}
        &\miint{Q_\rho^\lambda(z_0)} H(z,|Du|)\, dz\leq c\left(\miint{Q_{2\rho}^\lambda(z_0)} H(z,|Du|)^\theta\,dz\right)^\frac{1}{\theta},
    \end{aligned}
    $$
    whenever $Q_{2\kappa \rho}^\lambda (z_0)\subset \Omega_T$ satisfies \eqref{cond : p-phase condition}. Here, $$
    \theta_0=\begin{cases}
        \theta_0(n,p,q) \quad&\text{if \eqref{cond : main assumption with infty} holds},\\
        \theta_0(n,p,q,s) &\text{if \eqref{cond : main assumption with s} holds}.
    \end{cases}
    $$
\end{lemma}
\begin{proof}
    By Lemma \ref{lem : Caccioppoli inequality}, we get
    $$
    \begin{aligned}
    \miint{Q_\rho^\lambda(z_0)} H(z,|Du|)\, dz &\leq c \miint{Q_{2\rho}^\lambda(z_0)} \left[\frac{\Big|u-u_{Q_{2\rho}^\lambda(z_0)}\Big|^p}{(2\rho)^p}+a(z)\frac{\Big|u-u_{Q_{2\rho}^\lambda(z_0)}\Big|^q}{(2\rho)^q}\right] dz\\
    &\quad +c \lambda^{p-2} \miint{Q_{2\rho}^\lambda (z_0)} \frac{\Big|u-u_{Q_{2\rho}^\lambda(z_0)}\Big|^2}{(2\rho)^2}\, dz
    \end{aligned}
    $$
    for some $c=c(n,p,q,\nu,L)>1$. Using Lemmas \ref{lem : p-intrinsic parabolic Poincare inequality of p-term in p-intrinsic cylinder}, \ref{lem : p-intrinsic parabolic Poincare inequality of q-term in p-intrinsic cylinder}, \ref{lem : estimation of difference of u and mean u in p intrinsic cylinder with s<infty}, \ref{lem : estimation of difference of u and mean u in p intrinsic cylinder with s=infty} and Young's inequality yields
    $$
    \begin{aligned}
        &c\miint{Q_{2\rho}^\lambda(z_0)} \left[\frac{\Big|u-u_{Q_{2\rho}^\lambda(z_0)}\Big|^p}{(2\rho)^p}+a(z)\frac{\Big|u-u_{Q_{2\rho}^\lambda(z_0)}\Big|^q}{(2\rho)^q}\right] dz\\
        &\qquad \leq c_0\lambda^{p(1-\theta)}\miint{Q_{2\rho}^\lambda (z_0)} H(z,|Du|)^\theta \,dz\\
        &\qquad \leq \frac{1}{4}\lambda^p +c_0^{\frac{1}{\theta}}\theta[4(1-\theta)]^{\frac{1-\theta}{\theta}} \left(\miint{Q_{2\rho}^\lambda (z_0)} H(z,|Du|)^\theta \,dz\right)^\frac{1}{\theta}\\
        &\qquad \leq \frac{1}{4}\lambda^p +c_0^{\frac{1}{\theta_0}}4^{\frac{1-\theta_0}{\theta_0}}\left(\miint{Q_{2\rho}^\lambda (z_0)} H(z,|Du|)^\theta \,dz\right)^\frac{1}{\theta}
    \end{aligned}
    $$
    for any $\theta\in (\theta_2,1)$ and for some $c_0=c_0(\operatorname{data})>1$, where $\theta_2=\max\{\theta_1,\frac{q-1}{p}\}$ and $$\theta_1 =\begin{cases}
        \theta_1(n) \text{, which is defined in Lemma \ref{lem : estimation of difference of u and mean u in p intrinsic cylinder with s=infty}},\\
        \theta_1(n,p,q,s)\text{, which is defined in Lemma \ref{lem : estimation of difference of u and mean u in p intrinsic cylinder with s<infty}}.
    \end{cases}$$
    Write $c=c_0^{\frac{1}{\theta_0}}4^{\frac{1-\theta_0}{\theta_0}}$. Next, by \cite[Lemma 4.5]{2023_Gradient_Higher_Integrability_for_Degenerate_Parabolic_Double-Phase_Systems}, we obtain
    $$
    \miint{Q_{2\rho}^\lambda (z_0)} \frac{\Big|u-u_{Q_{2\rho}^\lambda(z_0)}\Big|^2}{(2\rho)^2}\, dz\leq c \lambda\left(\miint{Q_{2\rho}^\lambda(z_0)}\left[\frac{\Big|u-u_{Q_{2\rho}^\lambda(z_0)}\Big|^{\theta p}}{(2\rho)^{\theta p}}+|Du|^{\theta p}\right] dz\right)^\frac{1}{\theta p}
    $$
    for some $c=c(\operatorname{data})>1$ and for any $\theta\in (\frac{2n}{(n+2)p},1)$. It follows from Lemma \ref{lem : p-intrinsic parabolic Poincare inequality of p-term in p-intrinsic cylinder} and Young's inequality that
    $$
    \begin{aligned}
        &c\lambda^{p-2}\miint{Q_{2\rho}^\lambda (z_0)} \frac{\Big|u-u_{Q_{2\rho}^\lambda(z_0)}\Big|^2}{(2\rho)^2}\, dz\\
        &\quad \leq c \lambda^{p-1}\left(\miint{Q_{2\rho}^\lambda(z_0)}\left[\frac{\Big|u-u_{Q_{2\rho}^\lambda(z_0)}\Big|^{\theta p}}{(2\rho)^{\theta p}}+|Du|^{\theta p}\right] dz\right)^\frac{1}{\theta p}\\
        &\quad \leq \frac{1}{4}\lambda^p + c\left(\miint{Q_{2\rho}^\lambda (z_0)} H(z,|Du|)^\theta \,dz\right)^\frac{1}{\theta}.
    \end{aligned}
    $$
    By the third condition in \eqref{cond : p-phase condition}, we conclude that
    $$
    \miint{Q_\rho^\lambda(z_0)} H(z,|Du|)\, dz\leq c\left(\miint{Q_{2\rho}^\lambda (z_0)} H(z,|Du|)^\theta \,dz\right)^\frac{1}{\theta}
    $$
    for any $\theta \in (\theta_0,1)$, where $\theta_0=\max\{\theta_2,\frac{2n}{(n+2)p}\}$.
\end{proof}
Next, we obtain the following lemma in a similar way to \cite[Lemma 4.6]{2023_Gradient_Higher_Integrability_for_Degenerate_Parabolic_Double-Phase_Systems}.
\begin{lemma}\label{lem : reverse Holder inequality for p phase}
    Let $u$ be a weak solution to \eqref{eq : main equation} with \eqref{cond : main assumption with infty} or \eqref{cond : main assumption with s}. Then there exist constants $c=c(\operatorname{data})>1$ and $\theta_0\in(0,1)$ such that, for any $\theta\in(\theta_0,1)$,
    $$
    \iints{Q_{2\kappa\rho}^\lambda (z_0)} H(z,|Du|)\, dz\leq c\Lambda^{1-\theta}\iints{Q_{2\rho}^\lambda (z_0)\cap \Psi(c^{-1}\Lambda)} H(z,|Du|)^\theta \, dz.
    $$
    Here, $$
    \theta_0=\begin{cases}
        \theta_0(n,p,q) \quad&\text{if \eqref{cond : main assumption with infty} holds},\\
        \theta_0(n,p,q,s) &\text{if \eqref{cond : main assumption with s} holds}.
    \end{cases}
    $$
\end{lemma}

\subsection{The $(p,q)$-phase case.}
In this case, we consider the $(p,q)$-intrinsic cylinder denoted in \eqref{eq : definition of p,q-intrinsic cylinder}. We write
$$
S(u,G_\rho^\lambda(z_0))=\sup_{J_\rho^\lambda (t_0)}\dashint_{B_\rho (x_0)}\frac{\Big|u-u_{G_{\rho}^\lambda(z_0)}\Big|^2}{\rho^2}\, dx.
$$ 
\begin{lemma}
    Let $u$ be a weak solution to \eqref{eq : main equation} with \eqref{cond : main assumption with infty} or \eqref{cond : main assumption with s}. Then there exists a constant $c=c(n,p,q,\nu,L)>1$ such that
    $$
    S(u,G_{2\rho}^\lambda(z_0))\leq c\lambda^2,
    $$
    whenever $G_{2\kappa \rho}^\lambda (z_0)\subset \Omega_T$ satisfies \eqref{cond : p,q-phase condition}.
\end{lemma}
\begin{lemma}
    Let $u$ be a weak solution to \eqref{eq : main equation} with \eqref{cond : main assumption with infty} or \eqref{cond : main assumption with s}. Then there exists a constant $c=c(n,p,q)>1$ such that for any $\theta\in(\frac{n}{n+2},1)$,
    $$
    \begin{aligned}
        &\miint{G_{2\rho}^\lambda (z_0)} \left[\frac{\Big|u-u_{G_{2\rho}^\lambda(z_0)}\Big|^p}{(2\rho)^p}+a(z)\frac{\Big|u-u_{G_{2\rho}^\lambda(z_0)}\Big|^q}{(2\rho)^q}\right] dz\\
        &\qquad \leq c\miint{G_{2\rho}^\lambda (z_0)} \left[H_{z_0}^\theta\left(\frac{\Big|u-u_{G_{2\rho}^\lambda(z_0)}\Big|}{2\rho}\right)+H_{z_0}^\theta (|Du|)\right] dz\\
        &\qquad\qquad \times H_{z_0}^{1-\theta}\left(S(u,G_{2\rho}^\lambda (z_0))^\frac{1}{2}\right),
    \end{aligned}
    $$
    whenever $G_{2\kappa \rho}^\lambda (z_0)\subset\Omega_T$ satisfies \eqref{cond : p,q-phase condition}.
\end{lemma}
\begin{lemma}\label{lem : reverse Holder inequality for p,q phase}
    Let $u$ be a weak solution to \eqref{eq : main equation} with \eqref{cond : main assumption with infty} or \eqref{cond : main assumption with s}. Then there exist constants $c=c(n,p,q,\nu,L)>1$ and $\theta_0=\theta_0(n,p,q)\in(0,1)$ such that for any $\theta\in(\theta_0,1)$,
    $$
    \miint{G_\rho^\lambda(z_0)} H_{z_0}(|Du|)\, dz\leq c\left(\miint{G_{2\rho}^\lambda(z_0)}H_{z_0}^\theta(|Du|)\, dz\right)^\frac{1}{\theta},
    $$
    whenever $G_{2\kappa \rho}^\lambda(z_0)\subset \Omega_T$ satisfies \eqref{cond : p,q-phase condition}. Moreover, we have
    $$
    \iints{G_{2\kappa \rho}^\lambda (z_0)} H(z,|Du|)\, dz\leq c\Lambda^{1-\theta}\iints{G_{2\rho}^\lambda (z_0)\cap \Psi(c^{-1}\Lambda)} H(z,|Du|)^\theta\, dz.
    $$
\end{lemma}

\section{On the proof of our main results}
\subsection{\bf Vitali type covering lemma}\label{subsection 6.1}
To prove a Vitali type covering lemma for the collection of intrinsic cylinders, we write them as
$$
\mathcal{Q}(w)=\left\{
\begin{aligned}
    &Q^{\lambda_w}_{2\varrho_w}(w) &\text{if \eqref{case : p-phase} holds,}\\
    &G^{\lambda_w}_{2\varsigma_w}(w) &\text{if \eqref{case : p,q-phase} holds,}
\end{aligned}\right.
$$
where $\lambda_w,\, \varrho_w$ and $\varsigma_w$ are defined in Section \ref{section 3}. As in \cite[Subsection 5.2]{2023_Gradient_Higher_Integrability_for_Degenerate_Parabolic_Double-Phase_Systems}, we obtain a countable collection $\mathcal{G}$ of pairwise disjoint cylinders in $\mathcal{F}\coloneq \{\mathcal{Q}(w):w\in\Psi(\Lambda,r_1)\}$, where $\mathcal{G}$ satisfies the following two properties:
\begin{itemize}
    \item For each $\mathcal{Q}(z_1)\in\mathcal{F}$, there exists $\mathcal{Q}(z_2)\in\mathcal{G}$ such that 
    \begin{equation}\label{cond : intersect intrinsic cylinder}
        \mathcal{Q}(z_1)\cap\mathcal{Q}(z_2)\neq \emptyset
    \end{equation}
    \item For the above points $z_1$ and $z_2$, we get
    \begin{equation}\label{cond : relationship of ell_z_1 and ell_z_2}
        \ell_{z_1}\leq 2\ell_{z_2},
    \end{equation}
    where 
    $$
    \ell_{z_i}=\left\{
    \begin{aligned}
        &2\varrho_{z_i} &\text{if \eqref{case : p-phase} holds,}\\
        &2\varsigma_{z_i} &\text{if \eqref{case : p,q-phase} holds.}
    \end{aligned}\right.
    $$
\end{itemize}
Now, we claim that for the above points $z_1$ and $z_2$, there holds
\begin{equation}\label{cond : Vitali covering Q(z_1) and Q(z_2)}
    \mathcal{Q}(z_1)\subset \kappa\mathcal{Q}(z_2).
\end{equation}
We prove the claim by dividing it into the following four cases:
\begin{enumerate}[label=(\roman*)]
    \item\label{case : p-phases} $\mathcal{Q}(z_2)=Q_{\ell_{z_2}}^{\lambda_{z_2}}(z_2)\quad$ and $\quad\mathcal{Q}(z_1)=Q_{\ell_{z_1}}^{\lambda_{z_1}}(z_1)$,
    \item\label{case : p,q-phases} $\mathcal{Q}(z_2)=G_{\ell_{z_2}}^{\lambda_{z_2}}(z_2)\quad$ and $\quad\mathcal{Q}(z_1)=G_{\ell_{z_1}}^{\lambda_{z_1}}({z_1})$,
    \item\label{case : p,q-phase and p-phase} $\mathcal{Q}({z_2})=G_{\ell_{z_2}}^{\lambda_{z_2}}({z_2})\quad$ and $\quad\mathcal{Q}({z_1})=Q_{\ell_{z_1}}^{\lambda_{z_1}}({z_1})$,
    \item\label{case : p-phase and p,q-phase} $\mathcal{Q}({z_2})=Q_{\ell_{z_2}}^{\lambda_{z_2}}({z_2})\quad$ and $\quad\mathcal{Q}({z_1})=G_{\ell_{z_1}}^{\lambda_{z_1}}({z_1})$.
\end{enumerate}
According to \cite[Subsection 5.2]{2023_Gradient_Higher_Integrability_for_Degenerate_Parabolic_Double-Phase_Systems}, the claim holds in the spatial part by enlarging the radius of the set with respect to $z_2$ by factor $5$. To prove the claim for the time part, we need a comparison condition between $\lambda_{z_1}$ and $\lambda_{z_2}$. Thus, we show that if $\lambda_{z_1}\leq \lambda_{z_2}$, then
\begin{equation}\label{eq : lambda_z_2 less than lambda_z_1}
    \lambda_{z_2}\leq \left(2\left(1+K\right)\right)^\frac{1}{p}\lambda_{z_1}.
\end{equation}
To obtain a contradiction, suppose that \eqref{eq : lambda_z_2 less than lambda_z_1} is not satisfied. Then we get 
$$
\lambda_{z_1}^p<(1+K)\lambda^p_{z_1}<\frac{1}{2}\lambda_{z_2}^p \quad\text{and}\quad \lambda_{z_1}^q<\frac{1}{4}\lambda_{z_2}^q.
$$
If \eqref{case : p-phase} holds, then we obtain from \eqref{cond : Lambda=H(lambda)} that
$$
\Lambda=\lambda_{z_1}^p+a(z_1)\lambda_{z_1}^q\leq (K+1)\lambda_{z_1}^p<\frac{1}{2}\lambda_{z_2}^p\leq \frac{1}{2}\Lambda,
$$
which is a contradiction. If \eqref{case : p,q-phase} holds, we apply \eqref{cond : p,q-phase condition}$_2$ and \eqref{cond : Lambda=H(lambda)} to have
$$
\Lambda=\lambda_{z_1}^p+a(z_1)\lambda_{z_1}^q<\frac{1}{2}\lambda_{z_2}^p+\frac{1}{2}a(z_2)\lambda_{z_2}^q=\frac{1}{2}\Lambda,
$$
which is a contradiction. Thus, \eqref{eq : lambda_z_2 less than lambda_z_1} holds. Similarly, if $\lambda_{z_2}\leq \lambda_{z_1}$, then
$$
    \lambda_{z_1}\leq \left(2\left(1+K\right)\right)^\frac{1}{p}\lambda_{z_2}.
$$
To summarize, we have
\begin{equation}\label{eq : comparison condition of lambda_z_2 and _z_1}
    (4K)^{-\frac{1}{p}}\lambda_{z_1}\leq\lambda_{z_2}\leq (4K)^\frac{1}{p}\lambda_{z_1}.
\end{equation}

To prove the claim, we write $z_1=(x_{1},t_{1})$ and $z_2=(x_{2},t_{2})$ for $x_{1},x_{2}\in\mr^n$ and $t_{1},t_{2}\in\mr$.

\ref{case : p-phases} For any $\tau\in I^{\lambda_{z_1}}_{\ell_{z_1}}(t_{1})$, by \eqref{cond : relationship of ell_z_1 and ell_z_2} and \eqref{eq : comparison condition of lambda_z_2 and _z_1}, we have
$$
\begin{aligned}
    |\tau-t_{2}|&\leq |\tau-t_{1}|+|t_{1}-t_{2}|\leq 2\lambda_{z_1}^{2-p}\ell_{z_1}^2+\lambda_{z_2}^{2-p}\ell_{z_2}^2\\
    &\leq (32K+1) \lambda_{z_2}^{2-p}\ell_{z_2}^2\leq \lambda_{z_2}^{2-p}(\kappa\ell_{z_2})^2.
\end{aligned}
$$
Thus, we get $I^{\lambda_{z_1}}_{\ell_{z_1}}(t_{1})\subset \kappa I^{\lambda_{z_2}}_{\ell_{z_2}}(t_{2})$ and hence $Q_{\ell_{z_1}}^{\lambda_{z_1}}(z_1)\subset \kappa Q_{\ell_{z_2}}^{\lambda_{z_2}}(z_2)$.

\ref{case : p,q-phases} For any $\tau\in J^{\lambda_{z_1}}_{\ell_{z_1}}(t_{1})$, we have
$$
|\tau-t_{2}|\leq |\tau-t_{1}|+|t_{1}-t_{2}|\leq 2\frac{\lambda_{z_1}^2}{H_{z_1}(\lambda_{z_1})}\ell_{z_1}^2+\frac{\lambda_{z_2}^2}{H_{z_2}(\lambda_{z_2})}\ell_{z_2}^2.
$$
We apply \eqref{cond : Lambda=H(lambda)} and \eqref{eq : comparison condition of lambda_z_2 and _z_1} to have
$$
\frac{\lambda_{z_1}^2}{H_{z_1}(\lambda_{z_1})}=\frac{\lambda_{z_1}^2}{\Lambda}\leq 4K\frac{\lambda_{z_2}^2}{\Lambda}=4K\frac{\lambda_{z_2}^2}{H_{z_2}(\lambda_{z_2})}.
$$
By \eqref{cond : relationship of ell_z_1 and ell_z_2}, we obtain
$$
|\tau-t_{2}|\leq (32K+1)\frac{\lambda_{z_2}^2}{H_{z_2}(\lambda_{z_2})}\ell_{z_2}^2\leq \frac{\lambda_{z_2}^2}{H_{z_2}(\lambda_{z_2})} (\kappa \ell_{z_2})^2. 
$$
Thus, we have $J^{\lambda_{z_1}}_{\ell_{z_1}}(t_{1})\subset \kappa J^{\lambda_{z_2}}_{\ell_{z_2}}(t_{2})$ and hence $G_{\ell_{z_1}}^{\lambda_{z_1}}(z_1)\subset \kappa G_{\ell_{z_2}}^{\lambda_{z_2}}(z_2)$.

\ref{case : p,q-phase and p-phase} For any $\tau \in I^{\lambda_{z_1}}_{\ell_{z_1}}(t_{1})$, applying \eqref{cond : Lambda=H(lambda)}, we have
$$
|\tau-t_{2}|\leq |\tau-t_{1}|+|t_{1}-t_{2}|\leq 2\lambda_{z_1}^{2-p}+\frac{\lambda_{z_2}^2}{H_{z_2}(\lambda_{z_2})}\ell_{z_2}^2=2\lambda_{z_1}^{2-p}\ell_{z_1}^2+\frac{\lambda_{z_2}^2}{\Lambda}\ell_{z_2}^2.
$$
Since $\displaystyle \lambda_{z_1}^p\geq \sup_{Q_{10\varrho_{z_1}}(z_1)} a(\cdot)\lambda_{z_1}^q\geq a(z_1)\lambda_{z_1}^q$, by \eqref{cond : Lambda=H(lambda)} and \eqref{eq : comparison condition of lambda_z_2 and _z_1}, we get
$$
\lambda_{z_1}^{2-p}\leq \frac{2K\lambda_{z_1}^2}{\lambda_{z_1}^p+a(z_1)\lambda_{z_1}^q}\leq 8K^2\frac{\lambda_{z_2}^2}{\Lambda},
$$
and hence we see from \eqref{cond : relationship of ell_z_1 and ell_z_2} that
$$
|\tau-t_{2}|\leq (64K^2+1)\frac{\lambda_{z_2}^2}{\Lambda}\ell_{z_2}^2\leq \frac{\lambda_{z_2}^2}{\lambda_{z_2}^p+a(z_2)\lambda_{z_2}^q}(\kappa\ell_{z_2})^2.
$$
Thus, we get $I^{\lambda_{z_1}}_{\ell_{z_1}}(t_{1})\subset \kappa J^{\lambda_{z_2}}_{\ell_{z_2}}(t_{2})$ and hence $Q_{\ell_{z_1}}^{\lambda_{z_1}}(z_1)\subset \kappa G_{\ell_{z_2}}^{\lambda_{z_2}}(z_2)$.

\ref{case : p-phase and p,q-phase} For any $\tau\in J^{\lambda_{z_1}}_{\ell_{z_1}}(t_{1})$, using \eqref{cond : Lambda=H(lambda)} and \eqref{eq : comparison condition of lambda_z_2 and _z_1}, we obtain
$$
\begin{aligned}
    |\tau-t_{2}|&\leq |\tau-t_{1}|+|t_{1}-t_{2}|\leq 2\frac{\lambda_{z_1}^2}{H_{z_1}(\lambda_{z_1})}\ell_{z_1}^2+\lambda_{z_2}^{2-p}\ell_{z_2}^2\\
    &\leq 2\lambda_{z_1}^{2-p}\ell_{z_1}^2+\lambda_{z_2}^{2-p}\ell_{z_2}^2\leq (32K+1)\lambda_{z_2}^{2-p}\ell_{z_2}^2\\
    &\leq\lambda_{z_2}^{2-p}(\kappa\ell_{z_2})^2.
\end{aligned}
$$
Therefore, $J^{\lambda_{z_1}}_{\ell_{z_1}}(t_{1})\subset \kappa I^{\lambda_{z_2}}_{\ell_{z_2}}(t_{2})$ and $G_{\ell_{z_1}}^{\lambda_{z_1}}(z_1)\subset \kappa Q_{\ell_{z_2}}^{\lambda_{z_2}}(z_2)$. Thus, the claim is proved.

\subsection{\bf Proof of Theorems \ref{thm : main theorem with infty} and \ref{thm : main theorem with s}}\label{subsection 6.2}
Denote intrinsic cylinders in the countable pairwise disjoint collection $\mathcal{G}$ by, for $j=1,2,\cdots,$
$$
\mathcal{Q}_j \equiv \mathcal{Q}_j(z_j),
$$
where $z_j \in \Psi(\Lambda,r_1)$. Applying Lemmas \ref{lem : reverse Holder inequality for p phase} and \ref{lem : reverse Holder inequality for p,q phase}, we deduce that there exists $c=c(\operatorname{data})>1$ such that
$$
\miint{\kappa\mathcal{Q}_j} H(z,|Du|)\, dz\leq c\Lambda^{1-\theta}\miint{\mathcal{Q}_j\cap\Psi(c^{-1}\Lambda)}H(z,|Du|)^\theta\, dz
$$
for any $j \in\mathbb{N}$, where $\theta=\frac{\theta_0+1}{2} \in (0,1)$ with
$$
\theta_0=\begin{cases}
    \theta_0(n,p,q)\qquad & \text{if \eqref{cond : p-phase condition}$_1$ and \eqref{cond : main assumption with infty} hold},\\
    \theta_0(n,p,q,s)\qquad & \text{if \eqref{cond : p-phase condition}$_1$ and \eqref{cond : main assumption with s} hold},\\
    \theta_0(n,p,q)\qquad & \text{if \eqref{cond : p,q-phase condition}$_1$ holds}.\\
\end{cases}
$$
As in \cite[Subsection 5.3]{2023_Gradient_Higher_Integrability_for_Degenerate_Parabolic_Double-Phase_Systems}, by applying the Vitali covering lemma and Fubini's theorem, we obtain that for any $\varepsilon\in(0,\varepsilon_0)$,
\begin{equation}\label{eq : H estimate in Q_r to Q_2r}
    \miint{Q_r(z_0)} [H(z,|Du|)]^{1+\varepsilon}\, dz\leq c\Lambda_0^\varepsilon\miint{Q_{2r}(z_0)} H(z,|Du|)\, dz
\end{equation}
for some $c=c(\operatorname{data})>1$, $\varepsilon_0=\varepsilon_0(\operatorname{data})$ and $\Lambda_0$ is defined in \eqref{def : lambda_0 and Lambda_0}. Since $\lambda_0\geq 1$ and $2\leq p<q$, $\Lambda_0^\varepsilon\leq c \lambda_0^{\varepsilon q}$ for some $c=c(\operatorname{data},\|a\|_{L^\infty(\Omega_T)})$. Then the definition of $\lambda_0$ implies that
\begin{equation}\label{eq : Lambda_0 estimate}
    \Lambda_0^\varepsilon\miint{Q_{2r}(z_0)} H(z,|Du|)\, dz\leq c\left(\miint{Q_{2r}(z_0)} H(z,|Du|)\, dz\right)^{1+\frac{\varepsilon q}{2}}+c.
\end{equation}
Combining \eqref{eq : H estimate in Q_r to Q_2r} and \eqref{eq : Lambda_0 estimate} proves our main theorems. \qquad \qquad \qquad \qquad \qquad \qquad $\Box$

\bibliographystyle{abbrv}
\bibliography{ref}{}
\end{document}